\documentclass[a4paper,11pt]{amsart}

\usepackage{enumerate}
\usepackage{amsrefs}
\usepackage{dsfont}
\usepackage{amsmath}
\usepackage{amsthm}
\usepackage{amssymb}
\usepackage{amsfonts}
\usepackage{mathrsfs}  
\usepackage[shortlabels]{enumitem} 
\usepackage{color}
\usepackage[bookmarks=true,hyperindex,pdftex,colorlinks,citecolor=blue]{hyperref}
\usepackage{marginnote} 
\usepackage{graphicx} 
\usepackage{caption} 
\usepackage{soul} 

\newcommand{\tch}[1]{{\color{blue}{#1}}} 
\newcommand{\mnote}[1]{\marginnote{\scriptsize{#1}}} 

\newcommand{\dk}{d_{\K}^k}
\newcommand{\dK}{d_{\K}}
\newcommand{\Nk}{[\N]^k}

\newcommand{\K}{\mathbb{K}}
\newcommand{\N}{\mathbb{N}}
\newcommand{\Z}{\mathbb{Z}}
\newcommand{\R}{\mathbb{R}}

\newcommand{\M}{\mathbb{M}}

\newcommand{\J}{\mathcal{J}}

\newcommand{\F}{\mathcal{F}}
\newcommand{\Q}{\mathcal{Q}}

\newcommand{\ep}{\varepsilon}

\newcommand{\ov}[1]{\overline{#1}}
\newcommand{\norm}[1]{\|#1\|}

\newcommand{\Natural}{\mathbb N}

\newcommand{\Free}{\mathcal F}

\newcommand{\set}[1]{\left\{#1\right\}}
\newcommand{\restricted}{\mathord{\upharpoonright}}

\newcommand{\indicator}[1]{{\mathbf 1}_{{#1}}}

\newcommand{\duality}[1]{\left\langle#1\right\rangle} 
\newcommand{\closedball}[1]{B_{#1}} 
\newcommand{\sphere}[1]{S_{#1}} 

\newcommand{\eps}{\varepsilon}

\newcommand{\Lip}{\mathrm{Lip}}

\newcommand{\diam}{\text{diam}\,}

\newcommand{\n}{\overline n}
\newcommand{\m}{\overline m}

\newcommand{\cU}{\mathcal{U}}

\newcommand{\wlim}{\text{w-}\lim}
\newcommand{\wslim}{w^*\text{-}\lim}
\newcommand{\Sz}{\mathrm{Sz}}

\theoremstyle{plain}
\newtheorem{theorem}{Theorem}[section]
\newtheorem{lemma}[theorem]{Lemma}
\newtheorem{corollary}[theorem]{Corollary}
\newtheorem{proposition}[theorem]{Proposition}
\newtheorem{claim}{Claim}

\newtheorem*{T1}{Theorem~\ref{ThmAUCandPropQp}} 
\newtheorem*{T2}{Theorem~\ref{Szlenk}} 
\newtheorem*{T3}{Theorem~\ref{prop:Igraphdual}} 
\newtheorem*{T4}{Theorem~\ref{Thmc0SepDualDist}} 
\newtheorem*{T5}{Theorem~\ref{ThmCoarUniEmbWSCinDualSp}} 

\theoremstyle{definition}
\newtheorem*{definition*}{Definition}
\newtheorem{definition}[theorem]{Definition}

\newtheorem{problem}[theorem]{Problem}
\newtheorem*{acknowledgments}{Acknowledgments}
\theoremstyle{remark}
\newtheorem{remark}[theorem]{Remark}

\title[Interlaced graphs and embeddings into  dual Banach spaces]{On  Kalton's interlaced graphs and nonlinear embeddings into  dual Banach spaces}

\begin{document}

\author[B. M. Braga]{Bruno de Mendon\c{c}a Braga}

\address[B. M. Braga]{University of Virginia, 141 Cabell Drive, Kerchof Hall, P.O. Box 400137, Charlottesville, USA}
\email{demendoncabraga@gmail.com}
\urladdr{https://sites.google.com/site/demendoncabraga}

\author[G. Lancien]{Gilles Lancien}
\address[G. Lancien]{Laboratoire de Math\'ematiques de Besan\c con, Universit\'e Bourgogne Franche-Comt\'e, CNRS UMR-6623, 16 route de Gray, 25030 Besan\c con C\'edex, Besan\c con, France}
\email{gilles.lancien@univ-fcomte.fr}

\author[C. Petitjean]{Colin Petitjean}
\address[C. Petitjean]{LAMA, Univ Gustave Eiffel, UPEM, Univ Paris Est Creteil, CNRS, F-77447, Marne-la-Vall\'ee, France}
\email{colin.petitjean@u-pem.fr}

\author[A. Proch\'azka]{Anton\'in Proch\'azka}
\address[A. Proch\'azka]{Laboratoire de Math\'ematiques de Besan\c con,
Universit\'e Bourgogne Franche-Comt\'e,
CNRS UMR-6623,
16, route de Gray,
25030 Besançon Cedex, France}
\email{antonin.prochazka@univ-fcomte.fr}

\date{} 


\begin{abstract}
We study the nonlinear embeddability of Banach spaces and  the equi-embeddability of the family of Kalton's interlaced graphs $([\N]^k,d_\K)_k$ into dual spaces. Notably, we define and study a modification of Kalton's property $\mathcal Q$ that we call property $\mathcal{Q}_p$ (with $p \in (1,+\infty]$).
We show that if  $([\N]^k,d_\K)_k$ equi-coarse Lipschitzly embeds into $X^*$, then the Szlenk index of $X$ is greater than $\omega$, and that this is optimal, i.e., there exists a separable dual space $Y^*$ that contains $([\N]^k,d_\K)_k$ equi-Lipschitzly and  so that $Y$ has Szlenk index  $\omega^2$. We prove that $c_0$ does not coarse Lipschitzly embed into a separable dual space by a map with distortion strictly smaller than $\frac{3}{2}$.  We also show that neither $c_0$ nor $L_1$ coarsely embeds into a separable dual by a weak-to-weak$^*$ sequentially continuous map.
\end{abstract}



\maketitle


\section{Introduction}

It was a long standing problem in the nonlinear theory of Banach spaces whether every metric space {uniformly or} coarsely embeds into a reflexive Banach space (we refer the reader to Section~\ref{SectionPrel} for definitions)
{to which} a negative answer was only found  in 2007.
Indeed, N.~Kalton exhibited in \cite{Kalton2007Quarterly} a property for metric spaces, that he named property $\mathcal Q$, which serves as an obstruction to coarse embeddability into reflexive spaces (see Section~\ref{SectionPropQ} for precise statements).
Precisely, its absence is an obstruction to coarse embeddability into reflexive Banach spaces.
As it is easily checked, $c_0$ fails property $\mathcal Q$ and so it does not embed into any reflexive Banach space.
This property is defined in terms of the behaviour of
Lipschitz maps defined on a particular family of metric graphs: \textit{the Kalton's interlaced graphs} (see Section~\ref{s:KaltonGraphs}).

Furthermore, N. Kalton  proved the stronger result that any space $X$ coarsely containing $c_0$ must have some of its iterated duals nonseparable (see \cite{Kalton2007Quarterly}*{Theorem 3.6}).
Let us  point out that coarsely containing the James tree space $\mathcal{JT}$ would have the same impact on the iterated duals of $X$ \cite{LancienPetitjeanProchazka2018}*{Theorem 6.2}. The result of N. Kalton raises the following very natural problem:
\begin{problem}\label{ProblemUniversaln}
	Is there a universal $n\in\N$ so that if $c_0$ coarsely embeds into a Banach space $X$, then its $n$-th iterated dual $X^{(n)}$ is nonseparable?
\end{problem}
It is standard knowledge in the linear {(resp.\ non-linear)} theory of Banach spaces, that $c_0$ does not isomorphically {(resp.\ Lipschitz)} embed into any separable dual space. So it is also quite natural to wonder the following:	
\begin{problem}\label{ProbEmbc0SepDual}
Does $c_0$ coarsely embed into a  separable dual space?
\end{problem}
It is clear that a negative answer to this last problem would represent the strongest possible positive solution for Problem~\ref{ProblemUniversaln} (namely $n=2$). However, this last problem is still open even in the category of \textit{coarse-Lipschitz embeddings} (see Section~\ref{s:coarse} for a precise definition).
\begin{problem}\label{ProbEmbc0SepDual2}
	Does $c_0$ coarsely Lipschitz embed into a  separable dual space?
\end{problem} The current paper revolves around these questions. Therefore, inspired by N. Kalton, we not only study different notions of nonlinear embeddability  of $c_0$ into $X$, but we also analyse to which extent the equi-embeddability of the Kalton's interlaced graphs into a Banach space $X$ forces the dual of $X$ to be nonseparable.

We now describe the main findings of this paper. Throughout this paper, $[\N]^k$ denotes the
set
of all subsets of $\N$ with $k$ elements, $[\N]^{<\omega}$ denotes the
set
of all finite subsets of $\N$, and $d_\K=d_{\K,k}$ denotes Kalton's interlaced metric on $[\N]^k$. There exists a metric on $[\N]^{<\omega}$ which extends all metrics $d_{\K,k}$ simultaneously, and we also denote this metric  by $d_\K$ (see Subsection~\ref{s:KaltonGraphs} for precise definitions).

First of all,  inspired by \cite{KaltonRandrianarivony2008}, we define a modification of Kalton's property $\Q$ that we call property $\Q_p$ for $p\in (1,\infty]$. In a nutshell, while property~$\mathcal{Q}$ consists in a strong concentration inequality for Lipschitz maps $f : ([\N]^{k},d_{\K}) \to \R$ defined on the interlaced graphs, property $\Q_p$ is a concentration inequality proportional to $k^{1/p}$ (see Definition~\ref{PropertyQp}). In this way, property~$\mathcal{Q}$ may be seen as property $\Q_\infty$. It is readily seen that property $\Q_p$ is a coarse-Lipschitz invariant. The first main result relates this property with asymptotic uniform convexity (see Subsection~\ref{SubsectionAUSAUC} for definitions of asymptotic properties).
\begin{T1}
	Let $p\in (1,+\infty]$ and let $q\in [1,\infty)$ be the conjugate exponent of $p$. If a dual space $X^*$ admits an equivalent $q$-AUC$^*$ dual norm then $X^*$ has property $\mathcal{Q}_p$.
\end{T1}
We also prove that property~$\Q_{q}$ is equivalent to reflexivity for a certain class of Banach spaces (namely those having the $p$-alternating Banach-Saks property, see Corollary~\ref{cor:BanachSaks}). These results can be used to rule out the coarse-Lipschitz embeddability
between
certain Banach spaces (see Corollaries~\ref{cor:JamesQ_p} and \ref{optimalJamesQp}).

Next, recall that a
separable
Banach space $X$ has separable dual if and only if its Szlenk index $\mathrm{Sz}(X)$ is countable (see Subsection~\ref{SubsectionSzlenk} for a definition of the Szlenk index). Hence, $\mathrm{Sz}(X)$ can be seen as a quantitative measurement of ``how close to be nonseparable'' is $X^*$.
We obtain the following relation between containment of Kalton's interlaced graphs and the Szlenk index.
\begin{T2} Let $X$ be a Banach space. If the family of Kalton's interlaced graphs $(([\N]^{k},d_{\K}))_{k\in \N}$ equi-coarse Lipschitz embeds into $X^*$, then $\Sz(X) > \omega$, where $\omega$ denotes the first infinite ordinal.
\end{T2}

We also prove that if $X$ has summable Szlenk index then $X^*$ enjoys property~$\mathcal{Q}$. Moreover, Theorem~\ref{Szlenk} is actually optimal and the containment of Kalton's interlaced graphs cannot help us any further in the problem of whether $c_0$ coarsely embeds into a separable dual. Indeed, we show the following.
\begin{T3}
	The Kalton graph $([\N]^{<\omega}, \dK)$ Lipschitz embeds into a separable dual space $X^*$ with $\Sz(X) = \omega^2$.
\end{T3}

Although we were not able to obtain a negative answer to Problem \ref{ProbEmbc0SepDual2}, we obtained a restriction for the existence of a coarse Lispchitz embedding  from $c_0$ into $X^*$ based on the distortion of such embeddings. Before presenting our result, let us recall this definition. Let $X$ and $Y$ be Banach spaces and $f:X\to Y$ be a coarse Lipschitz embedding. We say that \emph{$f$ has coarse Lipschitz distortion strictly less than $K$} if there exist $A,B,C,D>0$ with $AC<K$ so that
\[\frac{1}{A}\|x-y\|-B\leq \|f(x)-f(y)\|\leq C\|x-y\|+D\]
for all $x,y\in X$.
We obtain the following.

\begin{T4}
If $c_0$ coarse Lipschitz embeds into a dual space $X^*$ with coarse Lipschitz distortion strictly less than $\frac{3}{2}$, then $X$ contains an isomorphic copy of $\ell_1$.
\end{T4}

In a different direction, we show that Problem \ref{ProbEmbc0SepDual} has a negative answer with the extra assumption that the embedding is weak-to-weak$^*$ sequentially continuous. Moreover, just as in the isomorphic theory, this also holds for the space $L_1$.

\begin{T5}
Neither $c_0$ nor $L_1$ can be coarsely (resp. uniformly) embedded into a separable dual Banach space by a map that is weak-to-weak$^*$ sequentially continuous.
\end{T5}

Since the ``weak-to-weak$^*$ sequential continuity'' hypothesis is not standard, a word on Theorem~\ref{ThmCoarUniEmbWSCinDualSp} is necessary. The first named author has begun  the study of coarse and coarse Lispchitz embeddings between Banach  spaces which also satisfy some continuity condition with respect to the weak topologies \cites{Braga2018IMRN,Braga2019Jussieu}. For instance, in contrast  to the famous result of I.~Aharoni that $c_0$ contains a Lipschitz copy of
every separable metric space \cite{Aharoni1974Israel}*{Theorem in page 288}, any Banach space not containing $\ell_1$ which can be coarsely embedded into $c_0$ by a weakly sequentially continuous map must actually be isomorphic to a subspace of $c_0$ \cite{Braga2019Jussieu}*{Theorem 1.6}. Also, although $\ell_p$ (resp. $L_p$) coarsely embeds into $\ell_q$ (resp. $L_q$) for all $p,q\in [1,2]$, the same is only true for weak sequentially continuous coarse embeddings $\ell_p\to \ell_q$ if $p\leq q$ \cite{Braga2018IMRN}*{Corollary 1.7}.
In particular, Theorem~\ref{ThmCoarUniEmbWSCinDualSp} shows that although the theory of coarse embeddability for members of the families $(\ell_p)_{p\in [1,2]}$ and    $(L_p)_{p\in [1,2]}$ are the same, this is not the case for weakly sequentially continuous embeddings. Indeed,  $L_1$ does not coarsely embed into $L_q$ by a weakly sequentially continuous map for any $q>1$, but $\ell_p$ does so into $\ell_q$ for all $q\geq p$.

This summarises our main findings. We now   give the definitions and terminology necessary for this paper.

\section{Preliminaries}\label{SectionPrel}

\subsection{Embeddings between metric spaces}\label{s:coarse}

Let $(M,d_M)$, $(N,d_N)$ be two metric spaces and $f \colon M \to N$ be a map. We define the \emph{compression modulus $\rho_f$} by letting
\[\rho_f (t) = \inf \{ d_N(f(x),f(y)) \, : \, d_M(x,y) \geq t \}\]
for each $t\geq 0$,  and the \emph{expansion modulus $\omega_f$} by letting
\[\omega_f (t) = \sup \{ d_N(f(x),f(y)) \, : \, d_M(x,y) \leq t \}\]
for all $t\geq 0$. We adopt the convention $\sup(\emptyset)=0$ and $\inf(\emptyset)=\infty$.
Note that for every $x,y \in M$,
\[\rho_f (d_M(x,y)) \leq d_N(f(x),f(y)) \leq \omega_f (d_M(x,y)).\]
Moreover, the  map $f:M\to N$ is called
\begin{enumerate}[(i)]
    \item a \emph{coarse embedding}  if $\lim_{t \to \infty} \rho_f (t) = \infty$ and  $\omega_f (t) < \infty$ for every $t \in [0,+\infty)$;
    \item a \emph{coarse Lipschitz embedding}  if there exits $A,B,C,D>0$ such that $\rho_f(t) \geq At-C$ and  $\omega_f (t) \leq Bt +D$ for every $t \in [0,+\infty)$;
    \item a \emph{Lipschitz embedding} if there exits $A,B>0$ such that $\rho_f(t) \geq At$ and  $\omega_f (t) \leq Bt $ for every $t \in [0,+\infty)$.
\end{enumerate}
Let $(M_i)_{i \in I}$ be a family of metric spaces. We say that the family $(M_i)_{i \in I}$ \emph{equi-coarsely embeds} (\emph{equi-coarsely Lipschitz embeds} and \emph{equi-Lipschitz embeds} respectively) into a metric space $N$ if there exist two maps
$$\rho, \, \omega \colon [0,+\infty) \to [0,+\infty)$$
and a family of maps $(f_i \colon M_i \to N)_{i \in I}$ such that
\begin{enumerate}
    \item $\rho(t) \leq \rho_{f_i}(t)$ for every $i \in I$ and $t \in [0,\infty) $,
    \item $\omega_{f_i}(t) \leq \omega(t)$ for every $i \in I$ and $t \in [0,\infty) $, and
    \item the maps $\rho$ and $\omega$ satisfy the properties (i) above (respectively (ii) for coarse Lipschitz embedding and (iii) for Lipschitz embedding).
\end{enumerate}

In order to refine the scale of coarse embeddings between Banach spaces, we will also shortly use the following notion. Let $X$ and $Y$ be two Banach spaces. We define $\alpha_Y(X)$ as the supremum of all $\alpha \in [0,1)$ for which there exists a coarse Lipschitz map $f:X\to Y$ (i.e., the expansion modulus $\omega_f$ is bounded above by an affine map) and $A,C$ in $(0,\infty)$ so that $\rho_f(t)\ge At^\alpha-C$ for all $t>0$. Then, $\alpha_Y(X)$ is called the {\it compression exponent of $X$ in $Y$}. Note that in the setting of Banach spaces, it is enough to impose that $\omega_f(a)<\infty$ for some $a>0$ to automatically get that $f$ is coarse Lipschitz. Indeed, decomposing any segment $[x,y]$ in $X$ into $\left\lfloor \frac{\|x-y\|}{a} \right\rfloor +1$ segments of length at most $a$, we obtain that $\|f(x)-f(y)\|\le \frac{\omega_f(a)}{a}\|x-y\|+\omega_f(a)$. This is more generally true when $X$ is a so-called metrically convex metric space.

\subsection{Kalton interlaced graphs}
\label{s:KaltonGraphs}
Given $k\in\N$ and an infinite $\M\subset \N$, let $[\M]^k$ denote the set of all strictly increasing $k$-tuples in $\M$. Given distinct $\bar n=(n_1,\ldots,n_k),\bar m=(m_1,\ldots,m_k)\in [\M]^k$, define a graph structure on $[\M]^k$ by declaring $\bar n\neq \bar m$ adjacent if and only if  either
\[n_1\leq m_1\leq n_2\leq \ldots\leq n_k\leq m_k\ \text{ or }\ m_1\leq n_1\leq m_2\leq \ldots\leq m_k\leq n_k.\]
The metric $d_\K^k$ is defined as the shortest path metric in the graph $[\M]^k$. The family $([\N]^k,d_\K^k)_k$ is the family of \emph{Kalton's interlaced graphs}.

For $k\in \N$  and $\M$ an infinite subset of $\N$, we put
$[\M]^{\le k}=\bigcup_{m\leq k}[\M]^m$, $[\M]^{< \omega}=\bigcup_{m\in\N}[\M]^m$ and  $[\M]^\omega=\{ S\subset \M: S\text{ is infinite}\} $. Just as in the finite case, the elements of $[\M]^\omega$ are always written as strictly increasing infinite tuples, i.e., if $\bar n=(n_1,n_2,\ldots)\in [\M]^\omega$, we always have $n_j< n_{j+1}$ for all $j\in\N$.

The distance $\dk$ is independent of the infinite subset of $\N$ chosen. So, given $k\in\N$, $\M_1\in[\N]^\omega$ and $\M_2\in [\M_1]^\omega$,  $[\M_2]^k$ is naturally  a metric subspace of $[\M_1]^k$. This is implied by the following proposition obtained in \cite{LancienPetitjeanProchazka2018}, which moreover  gives us an explicit formula to compute $\dk$.

\begin{proposition}[Proposition 2.2 of \cite{LancienPetitjeanProchazka2018}]
	Letting \[	d_\K(\n,\m) =\sup \Big\{\big| |\n\cap S| - |\m\cap S| \big| \; : \; S \text{ segment of } \N \Big\}\]
	for all $\n,\m\in [\N]^{<\omega}$, we have that $\dk=d_\K\restricted_{[\N]^k}$ for all $k\in\N$.\label{p:KaltonDistanceFormula}
\end{proposition}

The formula from the previous proposition also defines a graph metric on $[\N]^{<\omega}$ whose restriction to $[\N]^k$ of course coincides with $\dk$. From now on we simply denote the interlaced metric by $\dK$  (thus omitting the reference to $k$).

\begin{remark}\label{RemarkEquiCoarseLip}
It is easy to see that the sequence $([\N]^{k}, \dK)_k$ equi-coarsely Lipschitz embeds into a Banach space $X$ if and only if it equi-Lipschitz embeds into $X$. Indeed, this follows from the fact that, for any $k \in \N$, the map $f : ([\N]^k, \dK) \to ([\N]^{2k}, \frac12\dK)$  defined by:
	\[ \forall \, \n=(n_1,\ldots,n_k) \in [\N]^k, \quad f(\n) = (2n_1,2n_1+1, \ldots  , 2n_k, 2n_k +1) \]
is an isometry. Then, composing the isometric embedding of $([\N]^k, \dK)$ into $([\N]^{2^rk},2^{-r}\dK)$ for $r$ in $\N$ large enough and a rescaling of the equi-coarse Lipschitz embedding of $([\N]^{2^rk},\dK)$ into $X$ yields the conclusion.
\end{remark}

 For $\m=(m_1,m_2,\ldots, m_r)\in [\N]^{<\omega}$ and $\n=(n_1,n_2,\ldots, n_s)\in [\N]^{<\omega}$, we write $\m \prec \n$, if $r<s$ and  $m_i=n_i$, for $i=1,2,\ldots, r$, and we write $\m\preceq \n$ if $\m \prec \n$ or $\m=\n$. Thus $\m \preceq \n$ if $\m$ is an initial segment of $\n$. At last, for $\n=(n_1,\dots,n_k)$ and $\m=(m_1,\ldots,m_l)$ in $[\N]^{<\omega}$, we write $\n <\m$ if $n_k<m_1$.

\subsection{Szlenk index}\label{SubsectionSzlenk}
Let $X$ be  Banach space and $K$ be a weak$^*$ compact subset of $ X^*$. For each $\eps>0$, define
\[s_\eps(K)=K\setminus\{V\subset X^*\colon V\text{ weak$^*$ open and }\mathrm{diam}(V\cap K)< \eps\} . \]
Given an ordinal $\xi$,  $s_\eps^\xi(K)$ is defined inductively by letting $s_\eps^0(K)=s_\eps(K)$,  $s_\eps^{\xi+1}(K)=s_\eps(s^\xi_\eps(K))$ and   $s_\eps^{\xi}(K)=\cap_{\zeta<\xi}s^\zeta_\eps(K)$ if $\xi$ is a limit ordinal. We then define $\mathrm{Sz}(X,\eps)$ as the least ordinal $\xi$ so that $s_\eps^\xi(B_{X^*})=\emptyset$, if such ordinal exists, and  $\mathrm{Sz}(X,\eps)=\infty$ otherwise. The \emph{Szlenk index of $X$} is defined as
\[ \mathrm{Sz}(X)= \sup_{\eps>0}\mathrm{Sz}(X,\eps).\]

A Banach space $X$ is said to have \emph{summable Szlenk index} if there exists $c>0$ so that for all $\eps_1,\ldots,\eps_n>0$ the inequality
\[s_{\eps_n}(s_{\eps_{n-1}}(\ldots(s_{\eps_2}(s_{\eps_1}(B_{X^*}))\ldots ))\neq\emptyset\]
implies $\eps_1+\ldots+\eps_n\leq c$. It is known that any subspace of $c_0$ has summable Szlenk index, but the converse is not true (see details in Section~\ref{SectionConcentrationSzlenk}).

The Szlenk index of a Banach space is closely related to the behavior of the so-called weak$^*$-null or weak$^*$-continuous trees in its dual. So let us give the necessary definitions. For a Banach space $X$, we call \emph{tree of height $k$} in $X$ any family $(x(\n))_{\n\in[\N]^{\le k}}$, with $x(\n)\in X$. Then, if $\M \in [\N]^\omega$, $(x(\n))_{\n\in[\M]^{\le k}}$ will be called a \emph{full subtree} of $(x(\n))_{\n\in[\N]^{\le k}}$. For $\M \in [\N]^\omega$, a tree $(x^*(\n))_{\n\in[\M]^{\le k}}$ in $X^*$ is called \emph{weak$^*$-null} if for any $\n=(n_1,\dots,n_j) \in [\M]^{\le k-1}\setminus \{\emptyset\}$, the sequence $(x^*(n_1,\ldots,n_{j},t))_{t>n_{j},t\in \M}$ is weak$^*$-null and the sequence $(x^*(t))_{t\in \M}$ is also weak$^*$-null. It is called \emph{weak$^*$-continuous} if for any $\n=(n_1,\dots,n_j) \in [\M]^{\le k-1}\setminus \{\emptyset\}$, the sequence $(x^*(n_1,\ldots,n_{j},t))_{t>n_{j},t\in \M}$ is weak$^*$-converging to $x^*(n_1,\ldots,n_{j})$ and the sequence $(x^*(t))_{t\in \M}$ is also weak$^*$-converging to $x^*_\emptyset$. Then, the following proposition is a direct consequence of the definition of the Szlenk index.
\begin{proposition}\label{Szlenk-trees} Let $X$ be a Banach space and assume that $(x^*(\n))_{\n\in[\M]^{\le k}}$ is a weak$^*$-continuous tree in $B_{X^*}$ such that there exist $i_1<\cdots<i_l$ in $\{0,\ldots,k-1\}$ and $K_{i_1},\ldots,K_{i_l}>0$ satisfying
$$\forall s\in \{1,\ldots,l\}\ \forall \n \in [\M]^{i_s}\ \ \limsup_{t\to \infty, t\in \M}\|x^*(\n,t)-x^*(\n)\|\ge K_{i_s}.$$
Then
$$x^*_\emptyset \in s_{K_{i_l}}\ldots s_{K_{i_1}}(B_{X^*}).$$
\end{proposition}

\subsection{Asymptotic uniform smoothness and convexity}\label{SubsectionAUSAUC}

Let $X$ be a Banach space. We denote the set of all   closed subspaces of $X$ with finite codimension by $\mathrm{CoFin}(X)$. We define the  \emph{modulus of   asymptotic uniform smoothness  of $X$}  by letting
\[\overline{\rho}_X(t)=\sup_{x\in \partial B_{X}}\inf_{E\in \mathrm{CoFin}(X)}\sup_{y\in \partial B_{E}}\|x+ty\|-1\]
for each $t\geq 0$. The space $X$ is \emph{asymptotically uniformly smooth} (abbreviated by \emph{AUS}) if $\lim_{t\to 0^+}\overline{\rho}_X(t)/t=0$ for all $t>0$. Let $p\in (1,\infty]$. We say that  $X$ is  \emph{ $p$-asymptotically uniformly smooth} (abbreviated by \emph{$p$-AUS}) if there exists $C>0$ so that $\overline{\rho}_X(t)\leq Ct^p$ for all $t\in [0,1]$.

Let $X^*$ be a dual space. We denote the set of all weak$^*$ closed subspaces of $X^*$ with finite codimension by $\mathrm{CoFin}^*(X^*)$. We define the  \emph{modulus of weak$^*$ asymptotic uniform convexity of $X^*$}  by letting
\[\overline{\delta}_{X}^*(t)=\inf_{x^*\in \partial B_{X^*}}\sup_{E\in \mathrm{CoFin}^*(X^*)}\inf_{y^*\in \partial B_{E}}\|x^*+ty^*\|-1\]
for each $t\geq 0$. The space $X^*$ is \emph{weak$^*$ asymptotically uniformly convex} (abbreviated by \emph{AUC$^*$}) if $\overline{\delta}_X^*(t)>0$ for all $t>0$. Let $p\in [1,\infty)$. We say that  $X^*$ is  \emph{weak$^*$ $p$-asymptotically uniformly convex} (abbreviated by \emph{$p$-AUC$^*$}) if there exists $C>0$ so that $\overline{\delta}_X^*(t)\geq Ct^p$ for all $t\in [0,1]$.

We first recall the following classical duality result concerning these moduli (see  \cite{DKLR}*{Corollary 2.4}).

\begin{proposition}\label{duality} Let $X$ be a Banach space.
\begin{enumerate}[(i)]
\item Then $\|\ \|_X$ is AUS if and and only if $\|\ \|_{X^*}$ is AUC$^*$.
\item If $p\in (1,\infty]$ and $q\in [1,\infty)$ are conjugate exponents, then $\|\ \|_X$ is $p$-AUS if and and only if $\|\ \|_{X^*}$ is $q$-AUC$^*$.
\end{enumerate}
\end{proposition}

The next proposition is elementary.
\begin{proposition}\label{as-sequences}
For any weak$^*$-null sequence $(x^*_n)_{n=1}^\infty \subset X^*$ and for any $x^* \in X^*\setminus \set{0}$ we have
	\[
	\limsup_{n\to \infty} \norm{x^*+x^*_n} \geq \norm{x^*}\left(1+\overline{\delta}_X^*\left(\frac{\limsup_{n\to \infty} \norm{x_n^*}}{\norm{x^*}}\right)\right).
	\]
\end{proposition}

By iterating this estimate one can deduce the following property of weak$^*$-null trees in a $q$-AUC$^*$ dual space.
\begin{proposition}\label{qAUC-trees}
Let $X$ be a Banach space with a dual $q$-AUC$^*$ norm, for some $q\in [1,\infty)$. Then, there exists $c>0$ such that for any weak$^*$-null tree $(x^*(\n))_{\n\in[\N]^{\le k}}$ in $X^*$, there exists $\M \in [\N]^\omega$ such that
$$\forall \n\in [\M]^{k},\ \Big\|\sum_{\m \preceq \n}x^*(\m)\Big\|^q\ge c\sum_{\m \preceq \n}\|x^*(\m)\|^q.$$
\end{proposition}

This now standard fact could initially be found in \cite{KnaustOdellSchlumprecht1999}. See also  \cite{Kalton2013AsymptoticStructure} and  \cite{LancienPetitjeanProchazka2018}*{Lemma 3.6 and Lemma 3.7}.

We conclude this section by recalling the fundamental renorming result for spaces with Szlenk index equal to $\omega$. The result is due to H. Knaust, E.~Odell and Th. Schlumprecht \cite{KnaustOdellSchlumprecht1999} in the separable case and M. Raja \cite{Raja2013} in the non separable setting. The precise quantitative version can be found in \cite{GodefroyKaltonLancien2001}.

\begin{theorem}\label{SzlenkAUS} Let $X$ be a Banach space such that $\mathrm{Sz}(X)=\omega$. Then there exists $p\in (1,\infty)$ such that $X$ admits an equivalent $p$-AUS norm.
\end{theorem}

\subsection{General properties of Lipschitz maps into a dual space}\label{SubsectionGenProp}
We finish this preliminaries section by  gathering a few decompositon properties of Lipschitz maps from $([\N]^k,d_{\K})$ into a dual Banach space $X^*$ which will be heavily used throughout these notes. We start with an elementary separable reduction.

\begin{proposition}\label{toutcon} Let $X$ be a Banach space and  $f \colon [\N]^k \to X^*$ be a map. Then, there exists a separable subspace $Y$ of $X$ such that the closed linear span of $f([\N]^k)$ isometrically embeds into $Y^*$.
\end{proposition}

\begin{proof} Since $[\N]^k$ is countable, the closed linear span of $f([\N]^k)$ is a separable subspace of $X^*$; let us call it Z. Therefore, there exists a separable subspace $Y$ of $X$ such that
$$\forall x^* \in Z,\ \ \|x^*\|_{X^*}=\sup_{y\in B_Y} |x^*(y)|.$$
This concludes our proof.
\end{proof}

The next proposition is   \cite{LancienPetitjeanProchazka2018}*{Proposition 2.8}. As it is mentioned in \cite{LancienPetitjeanProchazka2018}, its proof follows the ideas of the proof of  \cite{BaudierLancienMotakisSchlumprecht2018}*{Lemma 4.1}. As usual $\Lip(f)$ denotes the best Lipschitz constant of a Lipschitz map $f$ between metric spaces; note that if $f:([\N]^k,d_{\K})\to Y$ with $Y$ being a normed vector space then $\Lip(f) = \omega_f(1)$.

\begin{proposition}\label{w*nulltree} Let $X$ be a separable Banach space,  $k\in\N$,  and $f:([\N]^k,d_{\K})\to X^*$  a Lipschitz map. Then there exist $\M \in [\N]^\omega$ and a weak$^*$-null tree $(x^*(\m))_{\m\in[\M]^{\le k}}$ in $X^*$ with $\|x^*(\m)\| \leq \Lip(f)$ for all $\m\in [\M]^{\leq k}\setminus\{\emptyset\}$ and so that
$$\forall \n \in [\M]^k,\ f(\n)=x^*_\emptyset + \sum_{i=1}^k x^*(n_1,\ldots,n_i)=\sum_{\m \preceq \n} x^*(\m).$$
\end{proposition}

Next, we will extract infinite subsets of $\M$ in order to simplify further the structure of $f$ restricted to the corresponding graph. So assume, for the sequel of this subsection, that $X$ is a separable Banach space, $f:([\N]^k,d_{\K})\to X^*$ is Lipschitz and $(x^*(\m))_{\m\in[\M]^{\le k}}$ is as in the conclusion of Proposition~\ref{w*nulltree}.

\begin{lemma}\label{Ramsey} Fix $\eps >0$. Then there exits $\M_1 \in [\M]^\omega$ such that for all $i\in\{1,\ldots,k\}$ there exists $K_i\in [0,\Lip(f)]$ satisfying
$$\forall (n_1,\ldots,n_i) \in [\M_1]^i,\ K_i\leq \|x^*(n_1,\dots,n_i)\|\le K_i+\eps.$$
\end{lemma}

\begin{proof} This a direct consequence of Ramsey's theorem and the compactness of $[0,\Lip(f)]$.
\end{proof}
Then we further extract in order to separate interlacing sequences. More precisely we show the following.
\begin{lemma}\label{Ramsey2}
There exists $\M_2 \in [\M_1]^\omega$ so that if $\M_2$ is enumerated as $\M_2=\{l_1<\cdots <l_n<\cdots\}$, then, for every $i \in \{1, \ldots , k\}$:
$$\forall (n_1,\ldots,n_i) \in [\N]^i,\quad  \|x^*(l_{2n_1}, \ldots ,l_{2n_i}) - x^*(l_{2n_1+1}, \ldots ,l_{2n_i+1})\| \geq \frac{K_i}{2}.$$
\end{lemma}

\begin{proof} Let us fix $i \in \{1, \ldots , k\}$. For $\m=(m_1,\ldots,m_{2i})\in [\M_1]^{2i}$, we denote $\m_{odd}=(m_1,m_3,\ldots,m_{2i-1})$ and $\m_{even}=(m_2,m_4,\ldots,m_{2i})$. It follows from  the weak$^*$-lower semi-continuity of $\|\cdot\|_{X^*}$ that for all $\M_2 \in [\M_1]^\omega$ there exists $\m \in [\M_2]^{2i}$ such that $\|x^*(\m_{odd})-x^*(\m_{even})\|\ge \frac{K_i}{2}$. Then, using Ramsey's theorem successively for each $i \in \{1, \ldots , k\}$, we get that there exists $\M_2 \in [\M_1]^\omega$ so that for every $i \in \{1, \ldots , k\}$ and every $\m \in [\M_2]^{2i}$, $\|x^*(\m_{odd})-x^*(\m_{even})\|\ge \frac{K_i}{2}$. Note now that if $(n_1,\ldots,n_i)\in [\N]^i$ and $\m=(l_{2n_1},\ldots,l_{2n_i+1}) \in [\M_2]^{2i}$, then $\m_{odd}=(l_{2n_1}, \ldots ,l_{2n_i})$ and $\m_{even}=(l_{2n_1+1}, \ldots ,l_{2n_i+1})$.      This finishes the proof.
\end{proof}

Then we set $y^*_{\emptyset}=0$ and for every $\n= (n_1,\ldots,n_i) \in [\N]^{\leq k}\setminus \{\emptyset\}$, we let
$$y^*(\n)  = x^*(l_{2n_1}, \ldots ,l_{2n_i}) - x^*(l_{2n_1+1}, \ldots ,l_{2n_i+1}).$$
We have that for every $\n \in \Nk$:
\[
 \begin{aligned}
  \Big\| \sum_{i=1}^{n} y^*(n_1,\ldots ,n_i) \Big\|&=
  \|f(l_{2n_1}, \ldots ,l_{2n_k}) - f(l_{2n_1+1}, \ldots ,l_{2n_k+1})\|\\
  &\leq \Lip(f).
  \end{aligned}
\]
Thus, we can build a weak$^*$-continuous tree $(z^*(\n))_{\n\in [\N]^{\leq k}}$ in $\Lip(f)B_{X^*}$ as follows:
$$\forall \n\in [\N]^{\leq k},\ \ z^*(\n)=\sum_{\m \preceq \n} y^*(\n).$$

We now summarize all these reductions in the following proposition.

\begin{proposition}\label{summarize} Let $X$ be a separable Banach space,  $k\in\N$,  $\eps>0$ and $f:([\N]^k,d_{\K})\to X^*$  a Lipschitz map. Then there exist $\M \in [\N]^\omega$, a weak$^*$-null tree $(x^*(\m))_{\m\in[\M]^{\le k}}$ in $X^*$ and constants $K_1,\ldots,K_k$ in $[0,\Lip(f)]$ such that
\begin{enumerate}[(i)]
\item For all $\m\in [\M]^{\leq k}\setminus\{\emptyset\}$, $\|x^*(\m)\| \leq \Lip(f)$.
\item For all $\n \in [\M]^k$, $f(\n)=\sum_{\m \preceq \n} x^*(\m).$
\item For all $i\in\{1,\ldots,k\}$ and all $(n_1,\ldots,n_i) \in [\M]^i$,
$$K_i\leq \|x^*(n_1,\dots,n_i)\|\le K_i+\eps.$$
\item Denote $\M=\{l_1,\ldots,l_n,\cdots\}$ with $l_1<\cdots <l_n<\cdots$, $y^*_{\emptyset}=z^*_{\emptyset}=0$ and,
for $\n= (n_1,\ldots,n_i) \in [\N]^{\leq k}\setminus \{\emptyset\}$,
$$y^*(\n)  = x^*(l_{2n_1}, \ldots ,l_{2n_i}) - x^*(l_{2n_1+1}, \ldots ,l_{2n_i+1})$$
and
$$z^*(\n)=\sum_{\m \preceq \n} y^*(\n).$$
Then $(z^*(\n))_{\n\in [\N]^{\leq k}}$ is a weak$^*$-continuous tree in $\Lip(f)B_{X^*}$ such that for every $i \in \{1, \ldots , k\}$ and every $(n_1,\ldots,n_i) \in [\N]^i$,
$$\|y^*(n_1,\dots,n_i)\| \geq \frac{K_i}{2}.$$
\end{enumerate}
\end{proposition}

\section{Property \texorpdfstring{$\mathcal{Q}_p$}{Qp}}\label{SectionPropQ}
N. Kalton proved in \cite{Kalton2007Quarterly}*{Theorem 3.6} that $c_0$ neither coarsely nor uniformly embeds into any Banach space $X$ whose  iterated duals are all separable. In the same paper, N. Kalton introduced the notion of property $\Q$ for a Banach space and showed that any reflexive Banach space has property $\Q$. Recall, a Banach space $X$ has \emph{property $\Q$}  if there exists $C\ge 1$ such that for every $k \in \N$ and every Lipschitz map $f \colon ([\N]^k,d_{\K}) \to X$, there exists an infinite subset $\M$ of $\N$ such that
\[  \|f(\overline{n})-f(\overline{m})\| \leq  C\omega_f(1)\]
 for all $\overline{n},\overline{m} \in [\M]^k$.

In this section, we introduce property $\Q_p$ for $p\in (1,\infty]$, which coincides with property $\Q$ when $p=\infty$. We then   give a sufficient condition for a Banach space to have property $\Q_p$ and use this in order to obtain some applications to the theory of nonlinear embeddings between Banach spaces.

\begin{definition}\label{PropertyQp}
Let $p \in (1,+\infty]$. We say that a Banach space $X$ has    \emph{property $\Q_p$} if there exists $C\ge 1$ such that for every $k \in \N$ and every Lipschitz map $f \colon ([\N]^k,d_{\K}) \to X$, there exists an infinite subset $\M$ of $\N$ such that
\[ \|f(\overline{n})-f(\overline{m})\| \leq  C\omega_f(1) k^{\frac{1}{p}}\]
 for all $\overline{n},\overline{m} \in [\M]^k$ (if $p=\infty$, we use the convention that $1/\infty=0$).
\end{definition}

Clearly, property $\Q_p$ implies property $\Q_q$ for all $q<p$. Hence, since every Banach space which either coarsely or uniformly embeds into a reflexive space has property $\Q$ \cite{Kalton2007Quarterly}*{Corollary 4.3}, the same holds for property $\Q_p$ for any $p\in (1,\infty]$.

The next proposition illustrates some simple permanence properties of  property $\Q_p$. Since its proof is immediate, we choose to omit it.

\begin{proposition}\label{PropositionPropQp}
Let $p\in (1,\infty]$ and let $X$ be a Banach space with property $\Q_p$. The following hold.
\begin{enumerate}[(i)]
\item If $Y$ coarse Lipschitz embeds into $X$, then $Y$ has property $\Q_p$.
\item If $\alpha_X(Y)=\alpha$, then for every $\varepsilon>0$ the space $Y$ has property $\Q_{p(\alpha-\varepsilon)}$.
\item The family $([\N]^k,d_{\K})_k$ does not equi-coarsely Lipschitz embed into $X$.
\item If $p=\infty$, then $([\N]^k,d_{\K})_k$ does not equi-coarsely   embed into $X$.
\end{enumerate}
\end{proposition}

The next theorem allows us to obtain new examples of spaces with property~$\Q_p$ and relates this property with asymptotic uniform convexity.

\begin{theorem}\label{ThmAUCandPropQp}
Let $X$ be a  Banach space and let $p\in (1,+\infty]$. Assume that $X$ admits an equivalent norm which is $p$-AUS  (or equivalently whose dual norm is $q$-AUC$^*$, where $q$ is the conjugate exponent of $p$). Then $X^*$ has property $\mathcal{Q}_p$.
\end{theorem}

\begin{proof} Assume, as it is allowed by Proposition~\ref{toutcon}, that $X$ is separable and that its norm is $p$-AUS. Therefore, the norm of $X^*$ is $q$-AUC$^*$, where $q$ is the conjugate exponent of $p$. Let $f:([\N]^k,d_{\K})\to X^*$ be a 1-Lipschitz map and fix $\eps>0$. Consider $\M\in [\N]^\omega$ and $(K_i)_{i=1}^k$  given by Proposition~\ref{summarize}. Since $(x^*(\m))_{\m\in[\M]^{\le k}}$ is a weak$^*$-null tree in $X^*$, it follows from Proposition~\ref{qAUC-trees} that we can find $n_1 < m_1 < \ldots < n_k < m_k$ in $\M$ so that we have the following lower $\ell_q$ estimate:
\begin{eqnarray*}
\| f(\n) - f(\m) \|^q &=& \Big\| \sum_{i=1}^k x^*(n_1, \ldots , n_k) - x^*(m_1, \ldots , m_k) \Big\|^q \\
&\geq&   c\left(\sum_{i=1}^k \norm{x^*(n_1, \ldots , n_k) }^q + \norm{x^*(m_1, \ldots , m_k) }^q \right),
\end{eqnarray*}
where $c>0$, only depends on the AUC$^*$ modulus of $X^*$.   Formally, we have applied Proposition~\ref{qAUC-trees} to the weak$^*$-null tree $(u^*(\m))_{\m\in[\M]^{\le 2k}}$ given by
\[u^*(n_1,\ldots,n_{l})=\left\{\begin{array}{l l }
x^*(n_1,n_3,\ldots,n_{l}),& \text{ if }l\text{ is odd}\\
-x^*(n_2,n_4,\ldots, n_l),& \text{ if }l \text{ is even}.
\end{array}\right.\]
 Since $f$ is 1-Lipschitz, we deduce that
$$ \sum_{i=1}^k K_i^q \leq \frac{1}{2c}.$$
Using H\"{o}lder's inequality and item $(iii)$ in Proposition~\ref{summarize}, this implies that for every $\n, \m \in [\M]^k$:
$$
\|f(\n) - f(\m)\| \leq 2 \sum_{i=1}^{k} K_i +2k\eps
\leq  \frac{2k^{1/p}}{(2c)^{1/q}} + 2k\eps .
$$
If $\eps$ was initially chosen small enough, this gives us the desired estimate.
\end{proof}

Let $p\in (1,\infty)$. We now recall the definition and some basic properties of the James space $\J_p$. We refer the reader to \cite{AlbiacKalton2006}*{Section 3.4} and references therein for more details on the classical case $p=2$. The James space $\J_p$ is the real Banach space of all sequences $x=(x(n))_{n\in \N}$ of real numbers with finite $p$-variation and verifying $\lim_{n \to \infty} x(n) =0$. The space $\J_p$ is endowed with the following norm
$$\|x\|_{\J_p} = \sup  \Big \{  \big (\sum_{i=1}^{k-1} |x(p_{i+1})-x(p_i)|^p \big )^{1/p}     \; \colon \; 1 \leq p_1 < p_2 < \ldots < p_{k} \Big \}. $$
This is the historical example, constructed for $p=2$ by R.C. James, of a quasi-reflexive Banach space which is isomorphic to its bidual. In fact $\J_p^{**}$ can be seen as the space of all sequences $x=(x(n))_{n\in \N}$ of real numbers with finite $p$-variation, which is $\J_p \oplus \R e$, where $e$ denotes the constant sequence equal to 1.

The standard unit vector basis $(e_n)_{n=1}^{\infty}$ is a monotone shrinking basis for $\J_p$. Hence, the sequence $(e_n^*)_{n=1}^{\infty}$ of the associated coordinate functionals is a basis of its dual $\J_p^*$.

 N. Kalton also proved that the James space $\J_2$ and its dual $\J_2^*$ fail property $\Q$ (see \cite{Kalton2007Quarterly}*{Proposition 4.7}). On the other hand, it is shown in \cite{LancienPetitjeanProchazka2018}*{Corollary 5.3} that the family $([\N]^k,\dk)_k$ does not equi-coarsely embed in $\J_p$, nor in $\J_p^*$ for any $p\in (1,\infty)$. It is known that, for $p\in (1,\infty)$,  $\J_p$ admits an equivalent $p$-AUS norm and $\J_p^*$ admits an equivalent $p'$-AUS norm, where $p'$ is the conjugate exponent of $p$ (see \cites{Lancien1994, Netillard2016}). Therefore we can state.

\begin{corollary} \label{cor:JamesQ_p}
Let $p\in (1,\infty)$ and $p'$ be its conjugate exponent. Then $\J_p$ has property $\Q_{p'}$ and $\J_p^*$ has property $\Q_p$.
\end{corollary}

A Banach space $X$ is said to have the \emph{alternating Banach-Saks property} if every bounded sequence $(x_n)_n$ in $X$ has a subsequence $(x_{n_j})_j$ so that its sequence of alternating-sign Ces\`{a}ro means $(\frac1k\sum_{j=1}^k(-1)^jx_{n_j})_k$ converges to $0$. N. Kalton proved in \cite{Kalton2007Quarterly}*{Theorem 4.5} that a Banach space with the alternating Banach-Saks property which also has property $\Q$ must be reflexive. We now present the $p$-version of this result. For that, we will need the following theorem, which is a version of  \cite{Kalton2007Quarterly}*{Theorem 4.4} to property $\Q_p$.

\begin{theorem}\label{clusterpoint}
Let $C\geq 1$, $p\in (1,\infty)$ and  $X$ be a Banach space  with  property $\Q_p$ with constant $C$. Then, for all $\eps>0$ and all bounded sequences $(x_n)_n$ in $X$ with weak$^*$ cluster point $x^{**}\in X^{**}$, there exists an infinite subset $\M$ of $\N$ so that
\[\Big\|\sum_{j=1}^{2k}(-1)^jx_{n_j}\Big\|\geq \frac{(1-\eps)}{C}d(x^{**},X)k^{1-1/p},\]
for all $k\in\N$ and all $n_1<\ldots<n_{2k}\in\M$.
\end{theorem}

\begin{proof}
If $x^{**}\in X$, the statement is trivial. Assume that  $\theta= d(x^{**},X)>0$. Let $B=\sup_{n\in\N}\|x_n\|$ and pick $\lambda>1$ and $\alpha \in (0,1)$ so that
\[ C^{-1}\lambda^{-2}\theta-  \alpha -2B\alpha\geq (1-\eps)C^{-1}\theta.\]

A classical argument due to James shows that, going to a subsequence of $(x_n)$, we can assume that
\begin{equation}\label{e:SeparatedConvexHulls}
\Big\|\sum_{j=1}^ka_jx_{n_j}-\sum_{j=1}^kb_jx_{m_j}\Big\|\geq \lambda^{-1}\theta,
\end{equation}
for all $k\in\N$, all $\bar n<\bar m\in [\N]^k$ and all $a_1,\ldots,a_k,b_1,\ldots,b_k\geq 0$ with $\sum_{j=1}^ka_j=\sum_{j=1}^kb_j=1$.

After extracting a further subsequence, we can also assume that
\begin{equation}\label{e:Mazur}
\Big\|\sum_{j=1}^{2l}(-1)^jx_{n_j}\Big\| \le \lambda\Big\|\sum_{j=1}^{2k}(-1)^jx_{n_j}\Big\|,
\end{equation}
for all $l<k\in \N$. Indeed, $0$ is in the weak$^*$-closure of the sequence $(x_n-x^{**})_n$, but not in its norm closure. This implies that $(x_n-x^{**})_n$ admits a $\lambda$-basic subsequence (see for instance Theorem 1.5.2 in \cite{AlbiacKalton2006}). Then (\ref{e:Mazur}) follows immediately after noticing that $\sum_{j=1}^{2l}(-1)^jx_{n_j}=\sum_{j=1}^{2l}(-1)^j(x_{n_j}-x^{**})$.

A simple application of Ramsey theory and a (skipped) diagonalization procedure gives an infinite subset $\M$ so that for all $k\in\N$, there exists $b_k>0$ with the following property:  for all $k\in\N$ and all $\alpha k^{1-1/p}\leq n_1<\ldots<n_{2k}$, we have that
\begin{equation}\label{e:ramsey}
\Big\|\sum_{j=1}^{2k}(-1)^jx_{n_j}\Big\|\in[b_k-\alpha, b_k].
\end{equation}

Fix $k\in\N$, let $\M_k=\{n\in\M\colon n\geq \alpha k^{1-1/p}\}$ and  define $f: ([\M_k]^k,d_{\K})\to X$ by setting
\[f(\bar n)=\sum_{j=1}^{k}x_{n_j},\]
for all $\bar n\in [\M_k]^k$. It follows from (\ref{e:Mazur}) and (\ref{e:ramsey}) that $\omega_f(1)\le \lambda b_k$. Let us mention that (\ref{e:Mazur}) is needed here because two elements of $[\M_k]^k$ at distance 1 for $d_{\K}$ are not always strictly interlaced and therefore (\ref{e:ramsey}) does not apply directly to estimate $\omega_f(1)$.

Since $X$ has property  $\Q_p$, there exist $\bar n<\bar m\in [\M_k]^k$ with
\[\|f(\bar n)-f(\bar m)\|\leq \lambda C b_kk^{1/p}.\]
Therefore, this and \eqref{e:SeparatedConvexHulls} give us that
\[b_k\geq C^{-1}\lambda^{-2}\theta k^{1-1/p}.\]

We now prove a lower estimate for all elements of $[\M]^{2k}$.
Fix $\bar n\in [\M]^{2k}$. Notice that $\lceil\alpha k^{1-1/p}\rceil\leq 2k$.
Let  $\bar m\in [\M_k]^{2k}$ be any element  so that $m_j=n_{\lceil\alpha k^{1-1/p}\rceil+j-1}$, for all $j\in \{1,\ldots,2k-\lceil\alpha k^{1-1/p}\rceil+1\}$. Then, we can pick $\beta\in\{-1,1\}$ so that
\begin{eqnarray*}
\Big\|\sum_{j=1}^{2k}(-1)^j x_{n_j}+\beta\sum_{j=1}^{2k}(-1)^jx_{m_j}\Big\|
&\leq &  \Big\|\sum_{j=1}^{\lceil\alpha k^{1-1/p}\rceil-1}(-1)^j x_{n_j}\Big\|\\
& & +\Big\|\sum_{j=2k-\lceil\alpha k^{1-1/p}\rceil+2}^{2k}(-1)^jx_{m_j}\Big\|\\
&\leq & 2B\alpha k^{1-1/p} .
\end{eqnarray*}

We conclude that
\begin{eqnarray*}
\Big\|\sum_{j=1}^{2k}(-1)^j x_{n_j}\Big\|&\geq &\Big\|\sum_{j=1}^{2k}(-1)^jx_{m_j}\Big\|- 2B\alpha k^{1-1/p}\\
&\geq& C^{-1}\lambda^{-2}\theta k^{1-1/p}-  \alpha -2B\alpha k^{1-1/p}\\
&\geq &(1-\eps)C^{-1}\theta k^{1-1/p}.
\end{eqnarray*}
\end{proof}

We now introduce a $p$-version of the alternating Banach-Saks property.

\begin{definition}
Let $p\in (1,\infty)$ and $C>0$. We say that $X$ has the \emph{$p$-alternating Banach-Saks property with constant $C>0$} if for all sequences $(x_n)_n$ in $B_X$ and all $k\in\N$, there exists an infinite subset $\M\subset \N$ so that
\[\Big\|\sum_{j=1}^k(-1)^jx_{n_j}\Big\|\leq C k^{1/p},\]
for all  $n_1<\ldots<n_k\in\M$.
\end{definition}

Notice that the $p$-alternating Banach-Saks property implies the alternating Banach-Saks property.

\begin{corollary} \label{cor:BanachSaks}
Let $p,q\in (1,\infty)$ be so that $q>p/(p-1)$ (i.e., $q$ is larger than the conjugate exponent of $p$). Let $X$ be a Banach space with the $p$-alternating Banach-Saks property and with property $\Q_{q}$. Then $X$ is reflexive.
\end{corollary}

\begin{proof}
Since reflexivity is separably determined, assume that $X$ is separable. Let $C\geq 1$ be so that $X$ has both the   $p$-alternating Banach-Saks property and  property $\Q_{q}$ with constant $C$. Suppose $X$ is not reflexive and pick $x^{**}\in B_{X^{**}}\setminus X$, so that $d(x^{**},X)>0$. Let $(x_n)_n$ be a sequence in $B_X$ with $x^{**}$ as a weak$^*$ cluster point.

Fix $k\in\N$. Since $X$ has  the  $p$-alternating Banach-Saks property with constant $C$, by going to a subsequence, we can assume that
\[\Big\|\sum_{j=1}^{2k}(-1)^jx_{n_j}\Big\|\leq 2^{1/p}C k^{1/p},\]
for all  $n_1<\ldots<n_{2k}\in\N$. Since $X$ has property $\Q_{q}$ with constant $C$ the previous theorem tells us that, by going to a subsequence, we can assume that
\[\Big\|\sum_{j=1}^{2k}(-1)^jx_{n_j}\Big\|\geq \frac{1}{2C}k^{1-1/q}d(x^{**},X),\]
for all $n_1<\ldots<n_{2k}\in\N$.

As $k$ was arbitrary, this shows that
\[\frac{1}{2C}k^{1-1/q}d(x^{**},X)\leq 2^{1/p}C k^{1/p}\]
for all $k\in\N$. As $1-1/q>1/p$, this gives us a contradiction.
\end{proof}

As another application of Theorem~\ref{clusterpoint}, we can show that Corollary \ref{cor:JamesQ_p} is optimal.

\begin{corollary}\label{optimalJamesQp} Let $p$ in $(1,\infty)$ and $p'$ be its conjugate exponent. Then, for any $r>p'$, $\J_p$ fails property $\Q_r$ and for any $s>p$, $\J_p^*$ fails property $Q_s$.
\end{corollary}

\begin{proof} We follow the proof of Proposition 4.7 in \cite{Kalton2007Quarterly}.

First, consider in $\J_p$ the sequence $(x_n)_n$ given by $x_n=\sum_{i=1}^ne_i$ for all $n\in\N$. We have that $(x_n)_n$ converges weak$^*$ to $e\in \J_p^{**}\setminus \J_p$. However, it is easy to see that there exists $C>0$ such that for any $n_1<\cdots<n_{2k}$, we have:
$$\Big\|\sum_{i=1}^{2k} (-1)^jx_{n_j}\Big\|_{\J_p}\le Ck^{1/p}.$$
For $r>p'$, according to Theorem~\ref{clusterpoint}, this prevents $\J_p$ from having property $\Q_r$.

We now consider in $\J_p^*$ the sequence $(e^*_n)_n$ which is weak$^*$-converging to an element $\lambda \in \J_p^{***}\setminus \J_p^*$, which is just the functional assigning its limit to any sequence of bounded $p$-variation. For $x\in \J_p$, we have
$$\Big|\langle \sum_{i=1}^{2k} (-1)^je^*_{n_j},x\rangle\Big| \le \sum_{j=1}^k|x(n_{2j})-x(n_{2j-1})|\le k^{1/p'}\|x\|_{\J_p}.$$
It follows that $\|\sum_{i=1}^{2k} (-1)^je^*_{n_j}\|_{\J_p^*}\le k^{1/p'}$. We then deduce from  Theorem~\ref{clusterpoint} that $\J_p^*$ fails property $Q_s$ for all $s>p$.
\end{proof}

Feeding Corollaries \ref{cor:JamesQ_p} and \ref{optimalJamesQp} into Proposition~\ref{PropositionPropQp}~(ii) we get information on some compression exponents of $\J_q$ in $\J_p$ or $\J_q^*$ in $\J_p^*$.
More precisely we have.
\begin{proposition} Let $p,q$ be in $(1,\infty)$ and $p',q'$ be their respective conjugate exponents.
\begin{enumerate}[(1)]
\item If $p<q$, then $\alpha_{\J_p}(\J_q)\le \frac{q'}{p'}$.
\item If $p>q$, then $\alpha_{\J_p^*}(\J_q^*)\le \frac{q}{p}$.
\end{enumerate}
\end{proposition}

Estimates on the compression exponents for the ``other half'' of the values of $p$ and $q$ are already known (see \cite{Netillard2016} or \cite{LancienRaja2017}). They are based on concentration properties for Lipschitz maps defined on the Hamming graphs with values in quasi-reflexive $p$-AUS spaces. When one wants to use asymptotic convexity as an obstruction for coarse Lipschitz embeddings, it is customary to use the so-called approximate midpoint principle (see for instance \cite{KaltonRandrianarivony2008}). However this method, as far as we know, only allows to show the impossibility of a coarse Lipschitz embedding, but does not provide extra information on the compression modulus. In fact, this method was used by F. Netillard \cite{Netillard2016} to prove that for $p<q$, $\J_q$ does not coarse Lipschitz embed in $\J_p$ and that for $p>q$, $\J_q^*$ does not coarse Lipschitz embed in $\J_p^*$. Our last corollary is an improvement of these results. This indicates that Theorem~\ref{ThmAUCandPropQp} can serve as an alternative to the approximate midpoint principle, but only in a non reflexive setting.

\section{Concentration properties and Szlenk indices}\label{SectionConcentrationSzlenk}

In this section, we obtain obstructions to the embeddability of Kalton's graphs into some dual Banach spaces.

\begin{theorem}\label{summable}
Let $X$ be a Banach space with summable Szlenk index. Then $X^*$ has property $\mathcal{Q}$.
\end{theorem}

\begin{proof} Assume, as it is allowed by Proposition \ref{toutcon}, that $X$ is separable. Let $f:([\N]^k,d_{\K})\to X^*$ be a 1-Lipschitz map, fix $\eps\in (0,\frac{1}{2k})$.
Let $\M\in [\N]^\omega$ and $K_1,\ldots,K_k$ be given by Proposition~\ref{summarize}.
Then it clearly follows from item $(iv)$ and Proposition~\ref{Szlenk-trees} that
$$0 \in s_{\frac{K_1}{2}}\ldots s_{\frac{K_k}{2}} (B_{X^*}).$$
Since the Szlenk index of $X$ is summable, we deduce that $\sum_{i=1}^k K_i \le 2C$, where $C$ is the ``summable Szlenk index constant'' of $X$.
Then, we deduce from items $(ii)$ and $(iii)$ of Proposition~\ref{summarize} that
$$\diam (f([\M]^k) \le 2\sum_{i=1}^k K_i+2k\eps \le 4C+1.$$
\end{proof}

\begin{remark}
Note  that Theorem~\ref{ThmAUCandPropQp} insures that if $X$ admits an equivalent norm whose dual norm is $1$-AUC$^*$ then $X^*$ has property $\mathcal{Q}$. It is known \cite{GodefroyKaltonLancien2000} that a separable Banach space admits an equivalent norm whose dual norm is $1$-AUC$^*$ if and only if $X$ is isomorphic to a subspace of $c_0$. It is an easy exercise to check that any subspace of $c_0$ has a summable Szlenk index. However, there are Banach spaces with summable Szlenk index that do not linearly embed into $c_0$. Before  describing a few of them, let us mention that a Banach space has a summable Szlenk index if and only if it is \emph{asymptotic-$c_0$} (see \cite{Causey2018}*{Theorem 4.1}). The original Tsirelson space, now denoted $T^*$, is an example of a reflexive asymptotic-$c_0$ space. Let us also mention that there exists a non reflexive quasi-reflexive Banach space which is asymptotic-$c_0$ (see Section 7 in \cite{BaudierLancienMotakisSchlumprecht2018} and references therein). In conclusion, Theorem~\ref{summable} applies to spaces that are not covered by N. Kalton's work nor by our Theorem~\ref{ThmAUCandPropQp}.
\end{remark}

\begin{theorem}\label{Szlenk} Let $X$ be a Banach space. If the family of Kalton's interlaced graphs $(([\N]^{<k},d_{\K}))_{k\in \N}$ equi-coarse Lipschitz embeds into $X^*$, then $\Sz(X) > \omega$.
\end{theorem}

\begin{proof}
We may assume again  that $X$ is separable (Proposition ~\ref{toutcon}). By Remark~\ref{RemarkEquiCoarseLip}, we can also assume that  $(([\N]^{<k},d_{\K}))_{k\in \N}$ equi-Lipschitz embeds into $X^*$. Hence, without loss of generality,    we may assume  that there exists $A\in (0,1]$ so that for any $k\in \N$ there exists $f_k : ([\N]^k,d_{\K}) \to X^*$ such that
$$\forall \n,\m \in [\N]^k,\ \ Ad_{\K}(\n,\m)\le \|f(\n)-f(\m)\| \le d_{\K}(\n,\m).$$
Let $\n < \m \in \Nk$ (that is such that $n_k<m_1$). By the triangle inequality we have
$$ \Big\| \sum_{\bar s \preceq\n} x^*(\bar s)   \Big\|  + \Big\| \sum_{\bar s \preceq \m} x^*(\bar s)   \Big\| \geq \|f(\n)- f(\m)\| \geq Ak.$$
For a fixed $k$ and a given $\eps>0\in (0,\frac{A}{4})$, we consider $\M\in [\N]^\omega$ given by Proposition~\ref{summarize}. It then follows from item $(iii)$ that
$2\sum_{i=1}^{k} K_i \geq Ak-2k\eps$.
Now, if we denote $I=\{1, \ldots,k\}$, $I_1 = \{ i\in I,\ K_i > \frac{A}{8}\} $ and  $N = |I_1|$, we have that
$$\frac{Ak}{2}-k\eps \leq \sum_{i=1}^{k} K_i = \sum_{I \setminus I_1} K_i + \sum_{I_1} K_i \leq \frac{A}{8}k + N.$$
From our choice of $\eps$, it follows that $N \geq \frac{Ak}{8}$.
Finally, we deduce from item $(iv)$ in Proposition~\ref{summarize} and Proposition~\ref{Szlenk-trees} that
$$0\in s_{\frac{A}{16}}^N(B_X^*)$$
and therefore that $\Sz(X,\frac{A}{16}) \geq \frac{Ak}{8}$. Since $k$ was arbitrary, this concludes our proof.
\end{proof}

\begin{remark}
As it is recalled in the introduction, a Banach space $X$ admits an equivalent AUS norm if and only if $\Sz(X)\le \omega$ and in that case there exists $p\in (1,\infty)$ such that $X$ admits an equivalent $p$-AUS norm. Therefore Theorem~\ref{ThmAUCandPropQp} is a quantitative version of Theorem~\ref{Szlenk}. In fact Theorem~\ref{Szlenk} is a consequence of Theorem~\ref{ThmAUCandPropQp} and these deep renorming results. We have chosen to present here an independent, self contained elementary proof.
\end{remark}

Let us now say that a Banach space $X$ has \emph{proportional Szlenk index} if there exists $C>0$ such that for all $\eps>0$, $\Sz(X,\eps) \le \frac{C}{\ep}$.
It is clear that a Banach space with summable Szlenk index has proportional Szlenk index. To the best of our knowledge, whether the converse implication is true is an open problem. We do not know either if the dual of a Banach space with a proportional Szlenk index has property $\mathcal{Q}$, but we can prove the following weaker concentration estimate.

\begin{proposition} Let $X$ be a Banach space with proportional Szlenk index.
Then, there exists $M>0$ such that for any $k\in \N$ and every Lipschitz map $f \colon ([\N]^k,d_{\K}) \to X^*$, there exists an infinite subset $\M$ of $\N$ such that:
$$\forall \, \overline{n},\overline{m} \in [\M]^k,\  \|f(\overline{n})-f(\overline{m})\| \leq  M(1+\log k)\Lip(f).$$
\end{proposition}

\begin{proof}By Proposition~\ref{toutcon}, we can  assume that $X$ is a separable  Banach space such that for all $\eps >0$, $\Sz(X,\eps) \le \frac{C}{\ep}$, for some $C>0$.
Let $f:([\N]^k,d_{\K}^k)\to X^*$ be a 1-Lipschitz map, fix $\eps>0$ and consider $\M\in [\N]^\omega$ and $K_1,\ldots,K_k$ given by Proposition~\ref{summarize}. Let $\eta=\frac{C}{k}$ and, for $r\in \N$, denote by $I_r$ the set of all $i$'s in $\{1,\ldots,k\}$ such that $2^{r-1}\eta\le K_i\le 2^r\eta$, and let $N_r$ be the cardinality of $I_r$.
It follows from item $(iv)$ in Proposition~\ref{summarize} and Proposition~\ref{Szlenk-trees} that $N_r\le \Sz(X,2^{r-2}\eta)\le \frac{4C}{2^r\eta}$ and so $N_rK_i\leq 4C$ if $i\in I_r$.
Notice also that, since $f$ is 1-Lipschitz,  $K_i\le 1$ for all $i\leq k$, which implies that $I_r$ is empty for $r>\log_2(\frac{k}{C})+1$. Let $N=\lceil \log_2(\frac{k}{C})+1\rceil$.
We deduce that
$$\sum_{i=1}^kK_i\le k\eta + \sum_{r=1}^{N}\sum_{i\in I_r}K_i\le k\eta + 4CN\le C+4CN.$$
Finally, using item $(iii)$ of Proposition~\ref{summarize}, we get that
$$\diam(f([\M]^k)\le 2C+8CN +2k\eps \le 3C+8CN,$$
if $\eps$ was initially chosen small enough.
In view of the definition of  $N$, this clearly yields the conclusion of our proposition.
\end{proof}


	\section{Optimality}\label{SectionOptimality}

In the previous section, we proved that if the family of Kalton's interlaced graphs equi-Lipschitz embeds into a dual Banach space $X^*$, then the Szlenk index of $X$ is at least $\omega^2$. Indeed, it is known that, when it is well defined, the Szlenk index of a Banach space is always of the form $\omega^\alpha$ for some ordinal $\alpha$ (see for instance \cite{Lancien2006}). Here we show that this result is optimal. That is, we exhibit a separable dual Banach space with Szlenk index $\omega^2$
which contains the interlaced graphs. To this aim we will use Lipschitz free spaces. Recall, if $(M,x_0)$ is a pointed metric space (a Banach space $X$  is always considered as a pointed metric space with $x_0=0$), then $\Lip_0(M)$ denotes the space of all  Lipschitz maps $f:M\to \R$  so that $f(x_0)=0$. Endowed with the norm $\|f\|=\Lip(f)$, $\Lip_0(M)$ is a Banach space. Given $x\in M$, the map $\delta_x:\Lip_0(M)\to \R$ given by $\delta_x(f)=f(x)$ for all $f\in \Lip_0(M)$ belongs to $\Lip_0(M)^*$, and we define the \emph{Lipschitz free space of $M$} as
\[\F(M)=\overline{\mathrm{span}}\{\delta_x\in \Lip_0(X)^*\colon x\in X\}.\]
We refer to the monograph 
\cite{CobzasMiculescuNicolae}
for the basic properties of $\F(M)$. Just note that the map $\delta:x\mapsto \delta_x$ is an isometry from $M$ into $\F(M)$.

In order to exhibit a separable dual Banach space with Szlenk index $\omega^2$, the strategy will be to consider the Lipschitz free space $\F(M)$ over a metric space $M$ which contains the interlaced graphs, and then  prove that $\F(M)$ has the required properties. In particular, the next corollary from \cite{GarciaLirolaPetitjeanProchazka2018Medi} will be useful for proving that $\F(M)$ is isometrically a dual Banach space.
In the following statement, $\mathcal C_\tau(M)$ stands for the set of maps from $M$ to $\R$ which are continuous with respect to some topology $\tau$ on $M$.

\begin{proposition}[Corollary 3.7 of \cite{GarciaLirolaPetitjeanProchazka2018Medi}]\label{prop:unifdiscduality}
	Let $(M, d)$ be a uniformly discrete, bounded, separable metric space with a distinguished point $0 \in M$. Assume that there is a Hausdorff topology $\tau$ on $M$ such that:
	\begin{itemize}
		\item[(i)]  $(M, \tau)$ is compact
		\item[(ii)] $d$ is $\tau$-lower semicontinuous.
	\end{itemize}
If $X=\mathrm{Lip}_0(M,d) \cap \mathcal C_\tau(M)$ is equipped with the Lipschitz norm $\norm{\cdot}_L$, then $X$ is an isometric predual of $\F(M)$. Moreover the weak$^*$-topology induced by $X$ on $\F(M)$ coincides with $\tau$ on $\delta(M)$, that we identify with $M$.
\end{proposition}

For any given $k\in\N$, a concrete bi-Lipschitz copy of the metric space $([\N]^{\leq k}, \dK)$ into $c_0$ is given by the map $f_k:[\N]^{\leq k}\to c_0$ defined by
$$\forall \ov{n} = (n_1 , \ldots ,n_j) \in [\N]^{\leq k}: \quad  f_k(\n)=\sum_{i=1}^j s_{n_i}, $$
where  $(s_n)_{n=1}^\infty$ stands for the summing basis of $c_0$. Indeed, one can easily check that \begin{equation} \label{eq:embc0} \tag{E}
 \forall \n,\m \in [\N]^{\leq k} : \quad  \frac12 \dK(\n,\m)\le \|f_k(\n)-f_k(\m)\| \le \dK(\n,\m)
 \end{equation}
 (see for instance \cite{LancienPetitjeanProchazka2018}*{Proposition 2.5}).

For each $k\in\N$,  let
$$ \widetilde{M_k}  = \ov{f_k\big([\N]^{\leq k}\big)}^{w^*}  \subset \ell_{\infty},$$
which is weak$^*$-compact since $f_k\big([\N]^{\leq k}\big)$ is bounded.
Letting $\mathds{1} \in \ell_{\infty}$ be the constant sequence equal to 1, it is readily seen that
$$  \widetilde{M_k}  = \Big\{\sum_{i=1}^j s_{n_i} + \ell \mathds{1} \: : \; j,\ell \in \N \cup \{0\}, \; j+\ell \leq k, \; n_1< \ldots < n_j \in \N \Big\}. $$
Hence, endowed with the usual norm $\|\cdot\|_\infty$ of $\ell_\infty$, the space $\widetilde{M_k}$ is  countable and uniformly discrete.
Finally, the norm $\|\cdot\|_\infty$ is trivially weak$^*$-lower semicontinuous.
The next corollary is therefore a   direct consequence of Proposition~\ref{prop:unifdiscduality}.

\begin{corollary} \label{cor:Mkdual}
For any $k\in\N$, the free space $\F\big(\widetilde{M_k} ,\|\cdot\|_\infty \big)$ is isometric to a separable dual Banach space $X_k^*$, where $X_k = \mathrm{Lip}_0(\widetilde{M_k},\|\cdot\|_\infty ) \cap \mathcal C_{w^*}(\widetilde{M_k})$.
\end{corollary}

\begin{theorem}\label{prop:Igraphdual}
	The Kalton graph $([\N]^{<\omega}, \dK)$ Lipschitz embeds into a separable dual space $X^*$ with $\Sz(X) = \omega^2$.
\end{theorem}

\begin{proof}
Let  $M = ([\N]^{<\omega}, \dK)$, and consider the distinguished point $0=\emptyset\in [\N]^{<\omega}$. For each $k\in\N$, let $$M_{2^k} = B(0,2^k) =[\N]^{\leq 2^k}.$$
So $M=\bigcup_{k\in\N} M_{2^k}$. Then we use Kalton's decomposition \cite{Kalton2004CollMath}*{Proposition 4.3} to deduce that for every $\ep>0$, $\F(M)$ $(1+\ep)$-linearly embeds into
$$\Big(\sum_{k \in \N} \F(M_{2^k},\dK) \Big)_{\ell_1}.$$

For each $k\in\N$, we let $\widetilde{M_{2^k}} $ be the metric subspace of $\ell_\infty$ as it is defined above. For each $k\in \N$,  let $X_{2^k}$ be the predual of  $\F(\widetilde{M_{2^k}} ,\|\cdot\|_\infty )$ given by  Corollary \ref{cor:Mkdual}.   It follows from (\ref{eq:embc0}) that for all $k\in\N$, $\F(M_{2^k},\dK)$ 2-linearly embeds into $X^*_{2^k}$. Since $M$ isometrically embeds into $\F(M)$, we deduce that, for any $\ep>0$,  $M$ Lipschitz embeds with distorsion $2(1+\ep)$ into
	$$\Big(\sum_{k \in \N} \F\big(\widetilde{M_{2^k}} ,\|\cdot\|_\infty \big) \Big)_{\ell_1} \equiv \Big(\sum_{k \in \N} X_{2^k} \Big)_{c_0}^*.$$
	
Let $X := (\sum_{k \in \N} X_{2^k} )_{c_0}$.
It remains to prove that $\Sz(X)=\omega^2$.
By Theorem~\ref{Szlenk} we know that $\Sz(X) > \omega$ and therefore $\Sz(X) \ge \omega^2$, so we only have to prove the reverse inequality.
Notice that \[X_{2^k} = \mathrm{Lip}_0(\widetilde{M_{2^k}},\norm{\cdot}_\infty ) \cap \mathcal C_{w^*}(\widetilde{M_{2^k}})\] equipped with its Lipschitz norm is isomorphic to a subspace of the Banach space
$(\mathcal C(\widetilde{M_{2^k}},w^*) , \|\cdot\|_{\infty})$ of continuous functions on the compact metrisable space
$(\widetilde{M_{2^k}},w^*)$.
Indeed, as $(\widetilde{M_{2^k}} , \|\cdot\|_{\infty})$ is bounded and uniformly discrete, we have that the sup-norm and the Lipschitz norm are equivalent on $X_{2^k}$.
Thus
\[X_{2^k}=\{f\in  \mathcal C_{w^*}(\widetilde {M_{2^k}})\colon f(0)=0\},\]
which clearly is a hyperplane of $\mathcal C_{w^*}(\widetilde{M_{2^k}})$ and it follows, for instance from \cite{AlbiacKalton2006}*{Proposition 4.4.1}, that $X_{2^k}$ is actually isomorphic to  $(\mathcal C(\widetilde{M_{2^k}},w^*) , \|\cdot\|_{\infty})$.

Next, we claim that the Cantor--Bendixson index of $\widetilde{M_{2^k}}$ is equal to $2^k+1$. Indeed it is readily seen by induction that the first $2^k$ derived sets are
	$$ \widetilde{M_{2^k}}^{(d)}  = \Big\{\sum_{i=1}^j s_{n_i} + \ell \mathds{1} \: : \; j,\ell \in \N \cup \{0\},\;  \ell \geq d, \; \; j+\ell \leq 2^k, \; n_1< \ldots < n_j \in \N \Big\},  $$
	whenever $d \in \set{1 , \ldots ,2^k}$ so the claim easily follows. This shows that $X_{2^k}$ is isomorphic to $c_0$ (e.g., \cite{AlbiacKalton2006}*{Theorem 4.5.2}) and therefore that $\Sz(X_{2^k}) = \omega$. Finally it follows from \cite{Brooker2011} that
	$$ \Sz(X) = \Sz\Big(\Big(\sum_{k \in \Z} X_{k} \Big)_{c_0}\Big) \leq \omega^2,$$
	and we are done.
\end{proof}

\begin{remark}
	The proof of the last proposition shows that $\F([\N]^{<\omega}, \dK)$ is isomorphic to a subspace of $X^* = \big(\sum_{k \in \N} \F\big(\widetilde{M_{2^k}} ,\|\cdot\|_\infty \big) \big)_{\ell_1}$. In fact, the image is even complemented in $X^*$. Indeed, this follows from the following two facts (we adopt the same notation as in the proof of Theorem~\ref{prop:Igraphdual} above). First, in Kalton's decomposition, the image of $\F(M)$ is complemented in $(\sum_k\F(M_{2^k}))_{\ell_1}$ (this is proved in detail in \cite{VectorValued}*{Proposition 3.5}).

	Second, we claim that $\big(\sum \F(M_{2^k}) \big)_{\ell_1}$ is isomorphic to a 1-complemented subspace of $\big(\sum \F(\widetilde{M_{2^k}}) \big)_{\ell_1}.$
	It is enough to show that $\F(M_{2^k})$ is 1-complemented in $\F(\widetilde{M_{2^k}})$ for every $k \in \N$, but this simply follows from the fact that
	\[
	\begin{array}{cccc}
		r_{2^k}  : & \widetilde{M_{2^k}} & \to & f_{2^k}(M_{2^k}) \\
		&\displaystyle \sum_{i=1}^j s_{n_i} + \ell \mathds{1} & \mapsto &\displaystyle  \sum_{i=1}^j s_{n_i}
	\end{array}
	\]
	is a 1-Lipschitz retraction.\qed
\end{remark}

\begin{remark}
	We have proved that, for every $\ep>0$, $([\N]^{<\omega}, \dK)$ Lipschitz embeds with distorsion $2(1+\ep)$ into a separable dual Banach space. It is actually possible to do it with distorsion $(1+\ep)$. To this end, instead of using  the natural embeddings of the $([\N]^{k}, \dK)$'s into $c_0$ (which are of distorsion 2), one can build concrete metric spaces containing isometrically the interlaced graphs and which satisfy the assumptions of Proposition~\ref{prop:unifdiscduality}. The counterpart is that one has to define by hand the required topology $\tau$ and then check that the distance is $\tau$-lower semicontinous (which was automatic with the $w^*$-topology in $\ell_\infty$). The same optimal estimate on the Szlenk index is obtained.

\end{remark}

\section{Low distortion embedding of the grid of \texorpdfstring{$c_0$}{c0} into duals}\label{SectionLowDistortion}

In this section, we produce two uniformly discrete countable metric spaces so that if they Lipschitz embeds into $X^*$ with Lipschitz distortion strictly less than $3/2$ or $2$, respectively, then $X$ must contain an isomorphic copy of $\ell_1$. We use this in order to prove Theorem~\ref{Thmc0SepDualDist}.

We define the \emph{integer grid of $c_0$} as
\[\mathrm{Grid}(c_0)=\{(x_n)_n\in c_0\colon \forall n\in\N,\ x_n\in\mathbb Z\}.\]
So $\mathrm{Grid}(c_0)$ is a $(1,1)$-net of $c_0$ (meaning that it is $1$-separated and for every $x\in c_0$, $d(x,\mathrm{Grid}(c_0))\le 1$). We consider it as a metric space with the metric inherited from $c_0$.

\begin{proposition}\label{Prop3/2}
 Let $X$ be a Banach space and $f:\mathrm{Grid}(c_0)\cap 2\closedball{c_0} \to X^*$ be a Lipschitz embedding with distortion strictly smaller than $\frac32$.  Then $X$ contains an isomorphic copy of $\ell_1$.
\end{proposition}

\begin{proof}
Replacing $f$ by $\lambda f$ for some appropriate $\lambda>0$, we may assume that there exists $D\in [1,\frac{3}{2})$ so that\[
\norm{x-y}\leq \norm{f(x)-f(y)}\leq D\norm{x-y}\]
for   all $x,y \in \mathrm{Grid}(c_0)\cap 2\closedball{c_0}$. Fix  $\varepsilon>0$ such that $3-\varepsilon-2D>0$.

Let $(e_n)_n$ be the canonical basis of $c_0$.
For every $k\in \Natural$, pick $x_k \in \sphere{X}$  such that \[\duality{x_k,f(2e_k)-f(-e_k)}\geq 3-\varepsilon.\]
We claim that the sequence $(x_k)_k$ has no weakly Cauchy subsequence. Indeed, let $\M=\set{m_1<m_2<\ldots}\in [\N]^\omega$  and  set $A_1=\set{m_{2k+1}:k\in \Natural}$ and $A_2=\Natural \setminus A_1$.
Then, for all $k\in\N$ and all  $m>m_{2k+1}$ we have that
\begin{eqnarray*}\duality{x_{m_{2k+1}},f(\indicator{A_1\cap [1,m]})-f(\indicator{A_2\cap [1,m]})}&=&
\Big\langle x_{m_{2k+1}},f(2e_{m_{2k+1}})-f(-e_{m_{2k+1}})\\ & &+f(\indicator{A_1\cap [1,m]})-f(2e_{m_{2k+1}})\\
& &+ f(-e_{m_{2k+1}}))-f(\indicator{A_2\cap [1,m]})\Big\rangle\\
&\geq& 3-\varepsilon-2D
\end{eqnarray*}
and
\begin{eqnarray*}
\duality{x_{m_{2k}},f(\indicator{A_1\cap [1,m]})-f(\indicator{A_2\cap [1,m]})}
&=&\Big\langle x_{m_{2k}},f(-e_{m_{2k}})-f(2e_{m_{2k}})\\ & &+f(\indicator{A_1\cap [1,m]})-f(-e_{m_{2k}})\\
& &+ f(2e_{m_{2k}})-f(\indicator{A_2\cap [1,m]})\Big\rangle\\
&\leq & -3+\varepsilon+2D.
\end{eqnarray*}
Let $\mathcal U$ be a nonprincipal ultrafilter on $\N$ and set \[x^*=w^*\text{-}\lim_{m,\mathcal U}\big(f(\indicator{A_1\cap [1,m]})-f(\indicator{A_2\cap [1,m]})\big).\]
The above inequalities imply that for all $k\in\N$:
$$x^*(x_{m_{2k+1}})\geq 3-\eps-2D\ \ \text{and}\ \  x^*(x_{m_{2k}})\leq -3+\eps+2D.$$
This shows that $(x_{m_k})_k$ is not weakly Cauchy.

By Rosenthal's $\ell_1$ theorem \cite{Rosenthal1974PNAS}, this implies that $(x_{k})_k$ has a subsequence equivalent to the standard unit basis of $\ell_1$. In particular, $X$ contains an isomorphic copy of $\ell_1$ and we are done.
\end{proof}

\begin{theorem}\label{Thmc0SepDualDist}
	If $c_0$ coarse Lipschitz embeds into a dual space $X^*$ with coarse Lipschitz distortion strictly less than $\frac{3}{2}$, then $X$ contains an isomorphic copy of $\ell_1$.
\end{theorem}

\begin{proof}
Assume $f$ is such a coarse Lipschitz embedding from $c_0$ into $X^*$. Replacing $f$ with $x\mapsto f(nx)/n$ for a large enough $n\in\N$, the map $f$ restricted to ${\mathrm{Grid}(c_0)\cap 2B_{c_0}}$ becomes a Lipschitz embedding with distortion strictly smaller than $3/2$. Then, it follows from  Proposition~\ref{Prop3/2} that $X$ contains an isomorphic copy of $\ell_1$.
\end{proof}

We will now show that replacing $\mathrm{Grid}(c_0)\cap 2\closedball{c_0}$ by an appropriate graph $M$, the distortion in Proposition~\ref{Prop3/2}  can  be pushed  up to 2. It is not clear if the metric space $M$ needed for this is isometric to a subset of $c_0$.

We define the graph $M$ as follows. Let  $S=[\Natural]^{<\omega}$,  $G=\N$ and $H=\N$. Moreover, we write $G=\{g_i\colon i\in\N\}$ and $H=\{h_i:i\in\N\}$, where $g_i=h_i=i$ for all $i\in\N$.  We define $M$ as the disjoint union
\[M=\set{0}\sqcup S \sqcup G \sqcup H\]
and we define a graph structure on $M$ by putting an edge
\begin{itemize}
\item between $0$ and any element of $S$,
\item between $A \in S$ and $g_k \in G$ iff $k \in A$, and
\item between $A \in S$ and $h_k \in H$ iff $k \notin A$.
\end{itemize}
Then endow  $M$  with the  shortest path distance, which we denote by $d_M$. It should be clear that $d_M(A,B)=2$ if $A\neq B \in S$ and that $d_M(g_k,h_k)=4$ for all $k\in\N$.

\begin{proposition}\label{PropDist2}
Let $f:(M,d_M) \to X^*$ be a Lipschitz embedding with distortion strictly less than 2.
Then $X$ contains $\ell_1$.
\end{proposition}

This proposition should be compared with \cite{ProchazkaSanchez}*{Theorem 3.1} where a metric space $(N,d)$ is constructed such that if $N$ Lipschitz embeds into $X$ with distortion strictly less than 2 then $X$ contains $\ell_1$.

\begin{proof}
Replacing $f$ by $\lambda f$ for some appropriate $\lambda>0$, we may assume that there exists $D\in [1,2)$ so that\[
\norm{x-y}\leq \norm{f(x)-f(y)}\leq D\norm{x-y}\]
for   all $x,y \in M$. Fix  $\varepsilon>0$ such that $4-\varepsilon-2D>0$.

For every $k\in \Natural$, let $x_k \in \sphere{X}$ be such that $\duality{x_k,f(g_k)-f(h_k)}\geq 4-\varepsilon$. We claim that $(x_k)_k$ does not contain any weakly Cauchy subsequence. Rosenthal's $\ell_1$ theorem thus implies that $X$ contains $\ell_1$. In order to prove our claim, let $\M=\set{m_1<m_2<\ldots}$ be an infinite subset of $\Natural$.
We set $A=\set{m_{2k+1}:k\in \Natural}$ and $B=\Natural \setminus A$.
Let $\cU$ be a nonprincipal ultrafilter on $\N$ and let $\xi_A=w^*\text{-}\lim_{m,\cU} f(A\cap[1,m])$
and $\xi_B=w^*\text{-}\lim_{m,\cU} f(B\cap[1,m])$.
We have
\begin{eqnarray*}
\duality{x_{m_{2k+1}},\xi_A-\xi_B}&=& \Big\langle x_{m_{2k+1}},f(g_{m_{2k+1}})-f(h_{m_{2k+1}})+\xi_A-f(g_{m_{2k+1}})\\& &+f(h_{m_{2k+1}})-\xi_B\Big\rangle\\ &\geq &4-\varepsilon -2D
\end{eqnarray*}
and
\begin{eqnarray*}
\duality{x_{m_{2k}},\xi_A-\xi_B}&=&\Big\langle x_{m_{2k}},f(h_{m_{2k}})-f(g_{m_{2k}})+\xi_A-f(h_{m_{2k}})\\
& &+f(g_{m_{2k}})-\xi_B\Big\rangle\\
&\leq &-4+\varepsilon +2D
\end{eqnarray*}
for all $k\in\N$. So $(x_{m_k})_k$ is not weakly Cauchy, and we are done.
\end{proof}

\begin{remark}
    The results of this section have a certain importance also for the theory of Lipschitz free spaces.
    Indeed, it is a well known open problem of Kalton (see \cite{Kalton2012} the remarks after Problem 1) to determine whether $\Free(M)$ enjoys the metric approximation property (MAP) for every bounded and uniformly discrete metric space $M$.
    One naive approach would be to show that for such $M$, the free space $\Free(M)$ is isometrically a dual. Then a stroke of Grothendieck theorem would imply that $\Free(M)$ has the (MAP) since, being isomorphic to $\ell_1$, it has the (AP).
    This approach does not work because not all such $\Free(M)$ are isometrically duals (other than the examples above see Example~5.8 in \cite{GarciaLirolaPetitjeanProchazka2018Medi}) and so Abraham Rueda Zoca proposed a refined strategy which consists in proving that for every bounded uniformly discrete metric space $(M,d)$ and for every $0<\alpha<1$ the free space $\Free(M,d^\alpha)$ of the $\alpha$-snowlaked $M$ is isometrically a dual.
    Then again by Grothendieck the space $\Free(M,d^\alpha)$ would enjoy the (MAP) and by approximation as  $\alpha$ tends to $1$, so would $\Free(M)$.
    Strike number two: this approach does not work either since for the metric space from Propostion~\ref{PropDist2} there is $0<\alpha_0<1$ such that for all $\alpha_0<\alpha<1$ the space $\Free(M,d_M^\alpha)$ is not isometrically a dual. Indeed, it is enough to take $\alpha_0$ such that the Banach-Mazur distance of $\Free(M)$ and $\Free(M,d_M^\alpha)$ is strictly less than $2$. Now, since $(M,d)$ embeds isometrically into $\Free(M)$, it will Lipschitz embed into $\Free(M,d_M^\alpha)$ with distortion $<2$ whenever $\alpha_0<\alpha<1$.
    Proposition~\ref{PropDist2} then implies that  $\Free(M,d_M^\alpha)$ cannot be a dual as it is separable.
    Nevertheless, one can check that the free spaces $\Free(M,d_M^\alpha)$ enjoy the (MAP).
\end{remark}

In the next section we will be dealing with a special kind of embeddings of $L_1$ into separable duals and so it is natural to ask whether an analogue of Proposition~\ref{Prop3/2} (resp.~\ref{PropDist2}) is true for some bounded uniformly discrete subsets of $L_1$.
The answer is negative.
Indeed, since $L_1$ is stable we can apply~\cite{BaudierLancien2015} to see that every such set embeds isometrically into some reflexive space.

\section{Weak sequentially continuous embeddings}\label{SectionWeakSeq}

In this section, we show that Problem \ref{ProbEmbc0SepDual} has a negative answer with the further assumption that the embedding is weak-to-weak$^*$ sequentially continuous.

Let $X$ be a separable Banach space, $K\subset X^*$ and $x^*\in K$. We say that \emph{$x^*$ is a point of weak$^*$-to-norm continuity of $K$} if  every sequence $(x^*_n)_n$ in $K$ which converges to $x^*$ in the weak$^*$ topology converges to $x^*$ in the norm topology.

The following is a well known application of the Baire category Theorem (see \cite{AlbiacKalton2006}*{Lemma 6.3.4}).

\begin{lemma}\label{LemmaWeakStarToNormContInCompDual}
Let $X$ be a Banach space with separable dual and $K$ be  a weak$^*$ compact subset of  $X^*$. Then $K$ has a point of weak$^* $-to-norm
continuity.
\end{lemma}

The next result should be compared with \cite{AlbiacKalton2006}*{Lemma 6.3.5}, which is a classical result in the isomorphic theory of Banach spaces.

\begin{lemma}\label{l:AimingAtPointOfWs-to-normContinuity}
Suppose $X$ and $Y$ are Banach spaces and assume that $Y^*$ is separable.
Let $f:B_X\to Y^*$ be a norm-to-weak$^*$ continuous bounded map so that its inverse exists and it is uniformly continuous.
Then every closed subset $F$ of $B_X$ contains a point $x$ such that $f(x)$ is a point of weak$^*$-to-norm continuity of $\overline{f(F)}^{w^*}$.
\end{lemma}

\begin{proof}
Let $F$ be a closed subset of $B_X$ and let $W$ be the weak$^*$ closure of $f(F)$.
Since $f(F)$ is bounded, $W$ is weak$^*$ compact.
By Lemma~\ref{LemmaWeakStarToNormContInCompDual}, there exists $y^*\in W$ which is a point of weak$^*$-to-norm continuity of $W$.
Let $(y^*_n)_n$ be a sequence in $f(F)$ converging to $y^*$ in the weak$^*$ topology.
By the choice of $y^*$, the sequence $(y^*_n)_n$ converges to $y^*$ in norm.
For each $n\in\N$, pick $x_n\in F$ so that $f(x_n)=y^*_n$.
Since $(y^*_n)_n $ is a Cauchy sequence in $Y^*$ and $f^{-1}$ is uniformly continuous, it follows that $(x_n)_n$ is a Cauchy sequence in $X$ and converges in norm to some  $x$.
As $F$ is closed, $x\in F$.
Since $f$ is norm-to-weak$^*$ continuous, we have \[f(x)=\wslim_nf(x_n)=\wslim_ny^*_n=y^*.\]
\end{proof}

\begin{lemma}\label{LemmaCoarseSeqWeakUnfContInvMap}
Let $X$ be either  $L_1$ or $c_0$, and $Y$ be a Banach space with separable dual. There is no weak-to-weak$^*$ sequentially continuous bounded map   $B_X\to Y$   with a uniformly continuous inverse.
\end{lemma}

\begin{proof}
Suppose $f:B_X\to Y^*$  is a weak-to-weak$^*$ sequentially continuous    bounded map with uniformly continuous inverse.
In particular $f$ is norm-to-weak$^*$ continuous, so using Lemma~\ref{l:AimingAtPointOfWs-to-normContinuity}, it is enough to find a closed $F\subset B_X$ such that for every $x \in F$ there is a sequence $(x_n)\subset F$ such that $x_n \to x$ weakly but not in norm.
Indeed, Lemma~\ref{l:AimingAtPointOfWs-to-normContinuity} furnishes a point $x\in F$ such that $f(x)$ is a point of weak$^*$-to-norm continuity of $f(F)$. We then have $f(x_n) \to f(x)$ weak$^*$ and therefore in norm.
Now the continuity of $f^{-1}$ yields $x_n \to x$ in norm which is impossible.

We define $F=\{x\in X\colon \|x\|\in [\frac12,1]\}$.
Let $x\in F$.
\begin{itemize}
 \item In the case $X=c_0$ we set $x_n=\frac{x+e_n}{\norm{x+e_n}}$.
It is clear that $x_n\in F$, $x_n\to x$ weakly and $\liminf\norm{x_n-x}\geq 1$.
 \item In the case $X=L_1$, let $(r_n)_n$ be the sequence of Rademarcher functions. Then $(r_nx)_n$ is weakly null. Moreover, it is easy to see that $\norm{x+r_nx}+\norm{x-r_nx}=2\norm{x}$ for all $n \in \N$.
 Using that $\liminf\norm{x\pm r_nx}\geq \norm{x}$, it follows that $\lim \norm{x+r_nx}=\norm{x}$.
 We can thus set $x_n=\frac{\norm{x}(x+r_nx)}{\norm{x+r_nx}}$ and we get again that $x_n\in F$, $x_n\to x$ weakly and $\liminf\norm{x_n-x}\geq \norm{x}\geq\frac12$.
\end{itemize}
\end{proof}

\begin{theorem}\label{ThmCoarUniEmbWSCinDualSp}
Neither $c_0$ nor $L_1$ can be coarsely (resp. uniformly) embedded into a separable dual Banach space by a map that is weak-to-weak$^*$ sequentially continuous.
\end{theorem}

Of course, it is well known that $L_1$ embeds uniformly and coarsely into $\ell_2$ (see~\cite{BenyaminiLindenstrauss} or~\cite{NaorICM}*{Corollary 3.1} for a simple explicit formula).

\begin{proof}
Let $X$ be either $L_1$ or $c_0$, $Y$ be a Banach space with separable dual and assume that there exists a weak-to-weak$^*$ sequentially continuous  map $f:X\to Y^*$   which is either a coarse  or a uniform embedding.

\begin{claim}
There exists a   coarse map $g:X\to \ell_2(Y^*)$ which is  weak-to-weak$^*$ sequentially continuous and so that $g^{-1}$ exists and   is uniformly continuous.
\end{claim}

\begin{proof}
If $f:X\to Y^*$  is a uniform embedding, there is nothing to be done. Indeed, we may simply take $g=i\circ f$, where $i:Y^*\to \ell_2(Y^*)$ is a linearly isometric inclusion.

Suppose  $f$ is a coarse embedding. Without loss of generality, assume that $f(0)=0$.  By \cite{Braga2017JFA}*{Lemma 5.1}, there exist  sequences of positive reals $(a_n)_n$ and $(b_n)_n$ so that the map $g:X\to \ell_2(Y^*)$ given by $g(x)=(f(a_nx)/b_n)_n$, for all $x\in X$,  is a well defined coarse embedding with uniformly continuous inverse.\footnote{Notice that  the hypothesis of  \cite{Braga2017JFA}*{Lemma 5.1} also demand the  map to be norm continuous. However, this is only used in order to guarantee that $g$ is norm continuous.} Since $f$ is weak-to-weak$^*$ sequentially continuous, so is $g$.
\end{proof}

Since $\ell_2(Y^*)$ is the dual of $\ell_2(Y)$ and separable, the result follows from Lemma~\ref{LemmaCoarseSeqWeakUnfContInvMap}.
\end{proof}

Notice that Theorem~\ref{ThmCoarUniEmbWSCinDualSp} implies that, for $p\in (1,\infty)$, $L_1$ does not coarsely (resp. uniformly) embed into either $\ell_p$ or $L_p$ by a weakly sequentially continuous map. In contrast with that, $\ell_q$ strongly embeds into $\ell_p$ by a weakly sequentially continuous map for all $q\leq p$ (see \cite{Braga2018IMRN}*{Theorem 1.8}).

Notice also that one could prove by a similar but simpler proof that if $X$ is a Banach space failing the point of continuity property (PCP) then $X$ cannot be coarsely (resp. uniformly) embedded into a separable dual Banach space $Y^*$ by a map that is weak-to-weak$^*$ continuous on bounded sets.
The added difficulty here is that, since we only have the sequential continuity, one has to present the weakly convergent sequences which furnish the contradiction.

\begin{acknowledgments}
Part of this paper was done while the first  named author was visiting the Universit\'{e}  Bourgogne Franche-Comt\'{e}, Besan\c{c}on, France  in May and July of 2019. This author would like to thank the hospitality of their mathematics department. The first named author  was partially supported by the Simons Foundation through York Science Fellowship. The three last named authors were supported by the French ``Investissements d'Avenir'' program, project ISITE-BFC (contract ANR-15-IDEX-03).
\end{acknowledgments}


\bibliographystyle{alpha}
\bibliography{bibliography2}

\begin{bibdiv}
\begin{biblist}

\bib{Aharoni1974Israel}{article}{
      author={Aharoni, I.},
       title={Every separable metric space is {L}ipschitz equivalent to a
  subset of {$c^{+}_{0}$}},
        date={1974},
        ISSN={0021-2172},
     journal={Israel J. Math.},
      volume={19},
       pages={284\ndash 291},
         url={https://doi.org/10.1007/BF02757727},
      review={\MR{0511661}},
}

\bib{AlbiacKalton2006}{book}{
      author={Albiac, F.},
      author={Kalton, N.},
       title={Topics in {B}anach space theory},
      series={Graduate Texts in Mathematics},
   publisher={Springer, New York},
        date={2006},
      volume={233},
        ISBN={978-0387-28141-4; 0-387-28141-X},
      review={\MR{2192298}},
}

\bib{BaudierLancien2015}{article}{
      author={Baudier, F.},
      author={Lancien, G.},
       title={Tight embeddability of proper and stable metric spaces},
        date={2015},
     journal={Anal. Geom. Metr. Spaces},
      volume={3},
      number={1},
       pages={140\ndash 156},
         url={https://doi.org/10.1515/agms-2015-0010},
      review={\MR{3365754}},
}

\bib{BaudierLancienMotakisSchlumprecht2018}{article}{
      author={{Baudier}, F.},
      author={{Lancien}, G.},
      author={{Motakis}, P.},
      author={{Schlumprecht}, Th.},
       title={{A new coarsely rigid class of Banach spaces}},
        date={published on line January 13 2020},
     journal={J. Inst. Math.  Jussieu},
     url={https://doi.org/10.1017/S1474748019000732},
       pages={},
      eprint={},
}




\bib{BenyaminiLindenstrauss}{book}{
      author={Benyamini, Y.},
      author={Lindenstrauss, J.},
       title={Geometric nonlinear functional analysis. {V}ol. 1},
      series={American Mathematical Society Colloquium Publications},
   publisher={American Mathematical Society, Providence, RI},
        date={2000},
      volume={48},
        ISBN={0-8218-0835-4},
      review={\MR{1727673}},
}

\bib{Braga2017JFA}{article}{
      author={Braga, B.~M.},
       title={Coarse and uniform embeddings},
        date={2017},
        ISSN={0022-1236},
     journal={J. Funct. Anal.},
      volume={272},
      number={5},
       pages={1852\ndash 1875},
         url={https://doi.org/10.1016/j.jfa.2016.12.005},
      review={\MR{3596709}},
}

\bib{Braga2018IMRN}{article}{
      author={Braga, B.~M.},
       title={{Nonlinear weakly sequentially continuous embeddings between
  Banach spaces}},
        date={2020},
        number={18},
        ISSN={1073-7928},
     journal={Int. Math. Res Not.},
     pages={5506\ndash 5533},
         url={https://doi.org/10.1093/imrn/rny181},
}

\bib{Braga2019Jussieu}{article}{
      author={Braga, B.~M.},
       title={On asymptotic uniform smoothness and nonlinear geometry of banach
  spaces},
        date={2021},
     journal={J. Inst. Math.  Jussieu},
     volume={20},
      number={1},
       pages={65 \ndash 102},
}

\bib{Brooker2011}{article}{
      author={{Brooker}, P. A.~H.},
       title={{Direct sums and the Szlenk index}},
        date={2011},
     journal={J. Funct. Anal.},
      volume={260},
      number={8},
       pages={2222\ndash 2246},
}

\bib{Causey2018}{article}{
      author={{Causey}, R.~M.},
       title={{Concerning summable Szlenk index}},
        date={2020},
     journal={Glasgow Math. Journal},
     volume={62},
      number={1},
       pages={59\ndash 73},
}

\bib{CobzasMiculescuNicolae}{book}{
      author={Cobza\c{s}, \c{S}.},
      author={Miculescu, R.},
      author={Nicolae, A.},
       title={Lipschitz functions},
      series={Lecture Notes in Mathematics},
   publisher={Springer, Cham},
        date={2019},
      volume={2241},
        ISBN={978-3-030-16488-1; 978-3-030-16489-8},
         url={https://doi.org/10.1007/978-3-030-16489-8},
      review={\MR{3931701}},
}

\bib{DKLR}{article}{
      author={Dilworth, S.~J.},
      author={Kutzarova, D.},
      author={Lancien, G.},
      author={Randrianarivony, N.~L.},
       title={Equivalent norms with the property {$(\beta)$} of {R}olewicz},
        date={2017},
        ISSN={1578-7303},
     journal={Rev. R. Acad. Cienc. Exactas F\'{\i}s. Nat. Ser. A Mat. RACSAM},
      volume={111},
      number={1},
       pages={101\ndash 113},
         url={https://doi.org/10.1007/s13398-016-0278-2},
      review={\MR{3596040}},
}

\bib{GarciaLirolaPetitjeanProchazka2018Medi}{article}{
      author={Garc\'{\i}a-Lirola, L.},
      author={Petitjean, C.},
      author={Proch\'{a}zka, A.},
      author={Rueda~Zoca, A.},
       title={Extremal structure and duality of {L}ipschitz free spaces},
        date={2018},
        ISSN={1660-5446},
     journal={Mediterr. J. Math.},
      volume={15},
      number={2},
       pages={Art. 69, 23~pp.},
         url={https://doi.org/10.1007/s00009-018-1113-0},
      review={\MR{3778926}},
}

\bib{VectorValued}{article}{
      author={Garc\'{\i}a-Lirola, L.},
      author={Petitjean, C.},
      author={Rueda~Zoca, A.},
       title={On the structure of spaces of vector-valued {L}ipschitz
  functions},
        date={2017},
        ISSN={0039-3223},
     journal={Studia Math.},
      volume={239},
      number={3},
       pages={249\ndash 271},
         url={https://doi.org/10.4064/sm8694-1-2017},
      review={\MR{3693571}},
}

\bib{GodefroyKaltonLancien2000}{article}{
      author={{Godefroy}, G.},
      author={{Kalton}, N.},
      author={{Lancien}, G.},
       title={{Subspaces of $c_0(\mathbb N)$ and Lipschitz isomorphisms}},
        date={2000},
     journal={Geom. Funct. Anal.},
      volume={10},
      number={4},
       pages={798\ndash 820},
}

\bib{GodefroyKaltonLancien2001}{article}{
      author={{Godefroy}, G.},
      author={{Kalton}, N.},
      author={{Lancien}, G.},
       title={{Szlenk indices and uniform homeomorphisms}},
        date={2001},
     journal={Trans. Amer. Math. Soc.},
      volume={353},
      number={10},
       pages={3895–\ndash 3918},
}

\bib{Kalton2004CollMath}{article}{
      author={Kalton, N.},
       title={Spaces of {L}ipschitz and {H}\"{o}lder functions and their
  applications},
        date={2004},
        ISSN={0010-0757},
     journal={Collect. Math.},
      volume={55},
      number={2},
       pages={171\ndash 217},
      review={\MR{2068975}},
}

\bib{Kalton2007Quarterly}{article}{
      author={Kalton, N.},
       title={Coarse and uniform embeddings into reflexive spaces},
        date={2007},
        ISSN={0033-5606},
     journal={Q. J. Math.},
      volume={58},
      number={3},
       pages={393\ndash 414},
         url={https://doi.org/10.1093/qmath/ham018},
      review={\MR{2354924}},
}

\bib{Kalton2012}{article}{
      author={Kalton, N.},
       title={The uniform structure of {B}anach spaces},
        date={2012},
        ISSN={0025-5831},
     journal={Math. Ann.},
      volume={354},
      number={4},
       pages={1247\ndash 1288},
         url={https://doi.org/10.1007/s00208-011-0743-3},
      review={\MR{2992997}},
}

\bib{Kalton2013AsymptoticStructure}{article}{
      author={Kalton, N.},
       title={Uniform homeomorphisms of {B}anach spaces and asymptotic
  structure},
        date={2013},
        ISSN={0002-9947},
     journal={Trans. Amer. Math. Soc.},
      volume={365},
      number={2},
       pages={1051\ndash 1079},
         url={https://doi.org/10.1090/S0002-9947-2012-05665-0},
      review={\MR{2995383}},
}

\bib{KaltonRandrianarivony2008}{article}{
      author={Kalton, N.},
      author={Randrianarivony, L.},
       title={The coarse {L}ipschitz geometry of {$l_p\oplus l_q$}},
        date={2008},
        ISSN={0025-5831},
     journal={Math. Ann.},
      volume={341},
      number={1},
       pages={223\ndash 237},
         url={https://doi.org/10.1007/s00208-007-0190-3},
      review={\MR{2377476}},
}

\bib{KnaustOdellSchlumprecht1999}{article}{
      author={Knaust, H.},
      author={Odell, E.},
      author={Schlumprecht, Th.},
       title={On asymptotic structure, the {S}zlenk index and {UKK} properties
  in {B}anach spaces},
        date={1999},
        ISSN={1385-1292},
     journal={Positivity},
      volume={3},
      number={2},
       pages={173\ndash 199},
         url={https://doi.org/10.1023/A:1009786603119},
      review={\MR{1702641}},
}

\bib{Lancien1994}{article}{
      author={{Lancien}, G.},
       title={{R\'eflexivit\'e et normes duales poss\'edant la propri\'et\'e de
  Kadec-Klee}},
        date={1993/94},
     journal={Pubications Math\'ematiques de Besan\c{c}on},
      volume={14},
}

\bib{Lancien2006}{article}{
      author={{Lancien}, G.},
       title={{A survey on the Szlenk index and some of its applications}},
        date={2006},
     journal={Rev. R. Acad. Cien. Serie A Mat.},
      volume={100},
      number={1-2},
       pages={209\ndash 235},
}

\bib{LancienRaja2017}{article}{
      author={{Lancien}, G.},
      author={{Raja}, M.},
       title={Asymptotic structure and coarse {L}ipschitz geometry of
  quasi-reflexive {B}anach spaces},
        date={2018},
     journal={Houston J. of Math},
      volume={44},
      number={3},
       pages={927\ndash 940},
}

\bib{LancienPetitjeanProchazka2018}{article}{
      author={Lancien, Gilles},
      author={Petitjean, Colin},
      author={Proch\'{a}zka, Antonin},
       title={On the coarse geometry of {J}ames spaces},
        date={2020},
        ISSN={0008-4395},
     journal={Canad. Math. Bull.},
      volume={63},
      number={1},
       pages={77\ndash 93},
         url={https://doi.org/10.4153/s0008439519000535},
      review={\MR{4059808}},
}

\bib{NaorICM}{inproceedings}{
      author={Naor, A.},
       title={{$L_1$} embeddings of the {H}eisenberg group and fast estimation
  of graph isoperimetry},
        date={2010},
   booktitle={Proceedings of the {I}nternational {C}ongress of
  {M}athematicians. {V}olume {III}},
   publisher={Hindustan Book Agency, New Delhi},
       pages={1549\ndash 1575},
      review={\MR{2827855}},
}

\bib{Netillard2016}{article}{
      author={Netillard, F.},
       title={Coarse Lipschitz embeddings of James spaces},
        date={2018},
     journal={Bull. Belg. Math. Soc. Simon Stevin.},
      volume={25},
      number={1},
       pages={71\ndash 84},
      review={\MR{3784507}},
}

\bib{ProchazkaSanchez}{article}{
      author={Proch\'{a}zka, A.},
      author={S\'{a}nchez-Gonz\'{a}lez, L.},
       title={Low distortion embeddings of some metric graphs into {B}anach
  spaces},
        date={2017},
        ISSN={0021-2172},
     journal={Israel J. Math.},
      volume={220},
      number={2},
       pages={927\ndash 946},
         url={https://doi.org/10.1007/s11856-017-1525-8},
      review={\MR{3666451}},
}

\bib{Raja2013}{article}{
      author={Raja, M.},
       title={On asymptotically uniformly smooth {B}anach spaces},
        date={2013},
        ISSN={0022-1236},
     journal={J. Funct. Anal.},
      volume={264},
      number={2},
       pages={479\ndash 492},
         url={https://doi.org/10.1016/j.jfa.2012.11.004},
      review={\MR{2997388}},
}

\bib{Rosenthal1974PNAS}{article}{
      author={Rosenthal, H.},
       title={A characterization of {B}anach spaces containing {$l^{1}$}},
        date={1974},
        ISSN={0027-8424},
     journal={Proc. Nat. Acad. Sci. U.S.A.},
      volume={71},
       pages={2411\ndash 2413},
         url={https://doi.org/10.1073/pnas.71.6.2411},
      review={\MR{0358307}},
}

\end{biblist}
\end{bibdiv}
	
\end{document}